\documentclass[11pt]{article}
\usepackage{amsmath}
\usepackage{amsthm}
\usepackage{amsfonts} 
\usepackage[pdftex]{graphicx} 
\usepackage[english]{babel} 
\usepackage[pdftex,linkcolor=black,pdfborder={0 0 0}]{hyperref} 
\usepackage{calc} 
\usepackage{dsfont}
\usepackage{bbm}
\usepackage{enumitem}

\usepackage{authblk}
\usepackage{physics}
\allowdisplaybreaks
\usepackage[a4paper, lmargin=0.1666\paperwidth, rmargin=0.1666\paperwidth, tmargin=0.1111\paperheight, bmargin=0.1111\paperheight]{geometry} 
\usepackage[all]{nowidow}
\usepackage[protrusion=true,expansion=true]{microtype}
\usepackage{lipsum} 
\hypersetup{ 	
	pdfsubject = {},
	pdftitle = {},
	pdfauthor = {}
}

\theoremstyle{plain}
\newtheorem{theorem}{Theorem}[section]

\newtheorem{lemma}[theorem]{Lemma}

\newtheorem*{Oliveira}{Oliveira's inequality}
\newtheorem*{Lecue}{Lecué and Mendelson's inequality}

\theoremstyle{definition}
\newtheorem{definition}[theorem]{Definition}

\theoremstyle{remark}

\newtheoremstyle{named}{}{}{\itshape}{}{\bfseries}{.}{.5em}{\thmnote{#3}#1}
\theoremstyle{named}

\def\floor#1{\left\lfloor #1 \right\rfloor}
\def\ceil#1{\left\lceil #1 \right\rceil}
\def\norm#1{\left\| #1 \right\|}
\def\Tr#1{\mathrm{Tr}\left( #1 \right)}

\def\IncompA#1{\mathrm{Incomp}_A\left(#1\right) }
\def\CompA#1{ \mathrm{Comp}_A \left( #1 \right) }

\title{Matrix anti-concentration inequalities with applications}

\author{Zipei Nie \thanks{Lagrange Mathematics and Computing Research Center, Huawei. Email: niezipei@huawei.com.}}

\begin{document}

	\maketitle	
	\begin{abstract}
		We provide a polynomial lower bound on the minimum singular value of an $m\times m$ random matrix $M$ with jointly Gaussian entries, under a polynomial bound on the matrix norm and a global small-ball probability bound $$\inf_{x,y\in S^{m-1}}\mathbb{P}\left(\left|x^* M y\right|>m^{-O(1)}\right)\ge \frac{1}{2}.$$ With the additional assumption that $M$ is self-adjoint, the global small-ball probability bound can be replaced by a weaker version. 
		
		We establish two matrix anti-concentration inequalities, which lower bound the minimum singular values of the sum of independent positive semidefinite self-adjoint matrices and the linear combination of independent random matrices with independent Gaussian coefficients. Both are under a global small-ball probability assumption. As a major application, we prove a better singular value bound for the Krylov space matrix, which leads to a faster and simpler algorithm for solving sparse linear systems. Our algorithm runs in $\tilde{O}\left(n^{\frac{3\omega-4}{\omega-1}}\right)=O(n^{2.2716})$ time where $\omega<2.37286$ is the matrix multiplication exponent, improving on the previous fastest one in $\tilde{O}\left(n^{\frac{5\omega-4}{\omega+1}}\right)=O(n^{2.33165})$ time by Peng and Vempala.
	\end{abstract}
	\section{Introduction}
	The study of extreme singular values of random matrices is a central topic in the non-asymptotic theory \cite{Rudelson2} of random matrices. In this paper we establish several new estimates of the minimum singular values of random matrices, with applications in compressed sensing, sparse linear system solving and sparse linear regression. Throughout the paper, let $\sigma_1(M)\ge \cdots\ge \sigma_{m}(M)$ denote all singular values of an $m\times m$ matrix $M$.
	
	\subsection{Minimum singular value of matrix with independent rows}
	Since its inception, random matrix theory  primarily dealt with matrices with mostly independent entries. In the simplest case where an $m\times m$ matrix $M$ has independent standard normal entries, Edelman \cite{Edelman} (see also \cite{Szarek} for a related work) proved that $$\mathbb{P}\left(\sigma_{m}\left(M\right)\le \varepsilon m^{-\frac{1}{2}}\right)\le \varepsilon$$
	for every $\varepsilon\ge 0$. Motivated by the interest in Bernoulli matrices, researchers considered this problem in more general settings. The first polynomial bound on the minimum singular value of any matrix with i.i.d. subgaussian entries was found by Rudelson \cite{Rudelson}, which only works for probability bound $\varepsilon\ge \Omega\left(m^{-\frac{1}{2}}\right)$. By the inverse Littlewood-Offord method introduced by Tao and Vu \cite{Tao3}, Rudelson and Vershynin \cite{Rudelson3} built an improved estimate for every $\varepsilon\ge 0$ up to a constant factor and an exponential small term. After a series of improvements \cite{Livshyts2,Rebrova,Rudelson4,Vershynin}, Livshyts, Tikhomirov and Vershynin proved that \cite{Livshyts}, for a random matrix $M$ with independent entries under a matrix norm bound and a uniform anti-concentration condition, there exists $c>0$ such that $$\mathbb{P}\left(\sigma_{m}\left(M\right)\le \varepsilon m^{-\frac{1}{2}}\right)\le c\varepsilon +\exp\left(-\frac{m}{c}\right)$$
	for every $\varepsilon\ge 0$.

	A slightly more general model assumes that an $n\times m$ ($n\ge m$) random matrix $X$ has independent rows. The minimum singular value $\sigma_{m}\left(X\right)$ of $X$ is the square root of the minimum singular value of the Gram matrix $X^T X$. When $X$ has i.i.d. rows $v_1^T, \ldots, v_n^T$ distributed as $v^T$, then $$\frac{1}{n} X^T X=\frac{1}{n} \sum_{i=1}^n v_i v_i^T$$ is the sample covariance matrix. This line of research (see \cite{Chafai,Koltchinskii2, Koltchinskii, Lecue, Srivastava, Tikhomirov,Oliveira, Yaskov1, Yaskov2, Yaskov3} for example) is motivated by the studies of random tensors, compressed sensing and high-dimensional statistics. Much work, such as the matrix Chernoff inequality \cite{Tropp} (see also \cite{Tropp2}), has been devoted to proving that $X^T X$ is close to $\mathbb{E}\left(X^T X\right)$ when $n$ is large, which implies bounds on $\sigma_1(X)$ and $\sigma_{m}(X)$ simultaneously. Koltchinskii and Mendelson \cite{Koltchinskii} pointed out that the roles of $\sigma_1(X)$ and $\sigma_{m}(X)$ are very different, and a bound on $\sigma_{m}(X)$ can be obtained without the concentration of $X^T X$. The following one-sided bound for sample covariance matrix was proved by Oliveira. 
	
	\begin{Oliveira}\cite{Oliveira}
		Let $M_1,\ldots, M_n$ be independent copies of an $m\times m$ positive semidefinite real symmetric random matrix $M$ with mean $\mathbb{E}\left(M\right)=\Sigma$. Suppose that $h$ is a positive real number such that $$\sqrt{\mathbb{E}\left(\left(x^T M x\right)^2\right)}\le h \cdot  \mathbb{E}\left(x^T M x\right)$$
		for every vector $x\in \mathbb{R}^m$. Then 
		$$\mathbb{P}\left(\sigma_{m}\left(\frac{1}{n}\sum_{i=1}^n \Sigma^{-\frac{1}{2}} M_i \Sigma^{-\frac{1}{2}}\right)<1-7h \sqrt{\frac{m+2\log \frac{2}{\delta}}{n}}   \right)\le \delta$$
		for every $0<\delta\le 1$. Especially, suppose that $$n\ge 1000 h^2 m,$$
		then we have $$\mathbb{P}\left(\sigma_{m}\left(\frac{1}{n}\sum_{i=1}^n \Sigma^{-\frac{1}{2}} M_i \Sigma^{-\frac{1}{2}}\right)\le \frac{1}{2}  \right)\le \exp\left(-\frac{n}{1000h^2}\right).$$
	\end{Oliveira}

	As a matter of fact, Oliveira's inequality is an anti-concentration result for general positive semidefinite real symmetric random matrices, with $M_i= v_i v_i^T$ being a particular case. It is worth noting that Zhivotovskiy \cite{Zhivotovskiy} recently developed a dimension-free version.
	
	The small-ball condition was introduced in \cite{Koltchinskii} as a reasonable assumption when bounding the minimum singular value of $X$. Lecué and Mendelson obtained the following bound under the weakest condition so far.
	
	\begin{Lecue} \cite{Lecue}
		There exists an absolute constant $c>0$ such that the following statement holds.
		
		Let $v_1,\ldots, v_n$ be independent copies of a random vector $v \in \mathbb{R}^m$. Suppose that $\alpha\ge 0$ and $0<\beta\le 1$ are two real numbers, such that the small-ball probability
		$$\mathbb{P}\left(\left|v^T x\right|>\alpha^{\frac{1}{2}} \right)\ge \beta$$
		holds for any $x$ in $S^{m-1}$. Suppose that $$n\ge \frac{cm}{\beta^2}.$$
		Then we have $$\mathbb{P}\left(\sigma_m\left(\frac{1}{n}\sum_{i=1}^m v_i v_i^T\right)\le \frac{\alpha\beta}{2}\right)\le \exp\left(-\frac{n\beta^2}{c}\right).$$
	\end{Lecue}

	We compare the above inequality with Oliveira's as follows. It follows from the Paley-Zygmund inequality that the assumption
	$$\sqrt{\mathbb{E}\left(\left(x^T M x\right)^2\right)}\le h \cdot  \mathbb{E}\left(x^T M x\right)$$
	implies 
	\begin{align*}
		&\inf_{x\in S^{m-1}}\mathbb{P}\left(x^T  \Sigma^{-\frac{1}{2}} M \Sigma^{-\frac{1}{2}} x>\frac{1}{2} \right)\\
		=&\inf_{x\in S^{m-1}}\mathbb{P}\left(x^T M x>\frac{1}{2} x^T \Sigma x\right)\\
		\ge &\inf_{x\in S^{m-1}} \frac{\left(\mathbb{E}\left(x^T M x\right)\right)^2}{4\cdot \mathbb{E}\left(\left(x^T M x\right)^2\right)}\\
		\ge& \frac{1}{4h^2}.
	\end{align*}
	Therefore when $\alpha, \beta$ and $h$ are treated as constants, the conclusions are equivalent up to constant factors. However, it is sometimes necessary (for example, when $X$ is sparse) to study the dependency on the parameter $\beta =\frac{1}{4h^2}$. For Oliveira's inequality, the dependency on $h$ is optimal. In contrast, Lecué and Mendelson's inequality has suboptimal dependency on $\beta$. Now we state our first main result.
	
	\begin{theorem} \label{main-1}
		There exists an absolute constant $c_{1.1}>0$ such that the following statement holds.
		
		Let $M_1,\ldots, M_n$ be independent $m\times m$ positive semidefinite real symmetric (resp. self-adjoint) random matrices. Suppose that $\alpha \ge 0$ and $0<\beta \le 1$ are two real numbers, such that the small-ball probability bound
		$$\sum_{i=1}^n \mathbb{P}\left(x^* M_i x>\alpha \right)\ge n\beta$$
		holds for any unit column vector $x$ in $\mathbb{R}^m$ (resp. in $\mathbb{C}^m$). Suppose that 
		$$n\ge \frac{c_{1.1}m}{\beta}\log\frac{2}{\beta}.$$
		Then we have
		$$\mathbb{P}\left(\sigma_m\left(\frac{1}{n}\sum_{i=1}^n M_i\right)\le \frac{\alpha\beta}{2} \right)\le \exp\left(-\frac{n\beta}{c_{1.1}}\right).$$
	\end{theorem}

	On the one hand, Theorem \ref{main-1} only assumes the small-ball probability bound which is weak and user-friendly, as Lecué and Mendelson's inequality does. On the other hand, it concerns the anti-concentration for general positive semidefinite self-adjoint matrices, and achieves optimal dependency on $\beta$ up to a log factor, as Oliveira's inequality does. 
	
	As direct corollaries, we can obtain parallel applications to \cite[Theorem 4.1]{Oliveira} and \cite[Theorem 5.2]{Oliveira} under weaker assumptions, in linear regression with random design and restricted eigenvalue constants respectively. Also, as an application in compressed sensing, we can improve \cite[Theorem A]{Lecue} on the number of measurements to ensure the exact reconstruction property when $\beta$ is not treated as a constant. Because these results depend on Oliveira's inequality or Lecué and Mendelson's inequality in a pretty direct way, we leave the precise statements and proofs to the interested readers.
	
	\subsection{Minimum singular value of matrix with jointly Gaussian entries}

	Bounding the smallest singular value for random matrices with non-trivial correlations among rows and columns is more challenging, even when a substantial amount of independence still exists. For example, a Rademacher random symmetric Bernoulli matrices was proved \cite{Costello} to be almost surely non-singular by Costello, Tao and Vu in 2005, while the corresponding case for random matrices with i.i.d. entries was settled \cite{Komlos} by Komlós almost 40 years prior. This non-singular probability bound leads to an exponential lower bound on the smallest singular value, while the first polynomial lower bound is obtained by Nguyen \cite{Nguyen} and Vershynin \cite{Vershynin} at the same time. 
	
	Two kinds of random matrices were intensively studied in the literature. For the first kind, the dependency among rows or columns is limited, so that the tools developed in the research of random matrices with i.i.d. entries can still be applied. Such examples include random symmetric matrices, adjacency matrices of random $d$-regular graphs, and matrices for which the correlation between two entries decays exponentially with respect to the distance between the entry positions. For the second kind, the matrices are highly structured. Such examples include random Toeplitz and Hankel matrices, and Haar distributed matrices such as circular ensembles.
	
	Instead, we focus on the key challenge of bounding the minimum singular value of a general $m\times m$ random matrix $M$ without the limited dependency condition or rich structures. We consider the simplest case where $M$ has jointly Gaussian entries. Such a matrix can be represented as $$M=M_0+\sum_{i=1}^{m^2} g_i M_i,$$
	where $M_0,M_1,\ldots, M_{m^2}$ are deterministic matrices and $g_1,\ldots, g_{m^2}$ are real independent standard normal random variables. In order to ensure that $M$ is almost surely non-singular, we need a condition to exclude the possibility that any linear combination of $M_0, M_1,\ldots, M_{m^2}$ is singular. The study of linear space of singular matrices \cite{Atkinson, Beasley,Dieu,Flanders,Seguins1,Seguins2} originates from Dieudonné's work \cite{Dieu}. The following theorem can be deduced from \cite[Lemma 1]{Flanders} by Flanders. Note that this result is rediscovered in a recent work \cite[Lemma 4.7]{Guo} by Guo, Oh, Zhang and Zorin-Kranich in harmonic analysis.
	
	\begin{theorem}\label{intro-non-singular-1}
		Let $S$ be a linear space of $m\times m$ real matrices. Suppose that there exists $M\in S$ with $x^* M y\neq 0$ for any $x,y\in S^{m-1}$. Then there exists $M\in S$ with $\det M\neq 0$.
	\end{theorem}
	
	For self-adjoint matrices, we relax our condition as follows.
	
	\begin{theorem}\label{intro-non-singular-2}
		Let $S$ be a linear space of $m\times m$ real symmetric matrices. Suppose that there exists $M\in S$ with $x^* M x\neq 0$ for any $x\in S^{m-1}$. Then there exists $M\in S$ with $\det M\neq 0$.
	\end{theorem}
	
	The corresponding statements for complex matrices also hold. 
	
	The interested readers can work out quantitative versions of Theorem \ref{intro-non-singular-1} and Theorem \ref{intro-non-singular-2}. The determinant argument leads to an exponential bound on $\sigma_m(M)$ assuming the global small-ball probability bound defined in the following way. 
 
	\begin{definition}\label{def-xy}
		For an integer $m$ and a matrix type $T$ which is either real, complex, real symmetric or self-adjoint, we define $S_{m,T}$ as the set of the pairs $(x,y)$ of unit $m$-dimensional unit column vectors with the following possible restrictions:
		\begin{enumerate}[label=(\arabic*)] 
			\item if $T$ is real or real symmetric, let $x$ and $y$ be real;
			\item if $T$ is real symmetric or self-adjoint, let $x=y$.
		\end{enumerate}
	\end{definition}
	
	\begin{definition}\label{def-gs}
		For an $m\times m$ random matrix $M$, a non-negative real number $\alpha$ and a matrix type $T$ which is either real, complex, real symmetric or self-adjoint, we define the global small-ball probability as $$P_{\alpha,T}(M):=\inf_{x,y}\mathbb{P}\left(\left|x^* M y\right|> \alpha \right),$$
		where $(x,y)$ ranges over $S_{m,T}$.
	\end{definition}

	Our second main result is a polynomial bound on the minimum singular value $\sigma_m(M)$ of an $m\times m$ random matrix $M$ with jointly Gaussian entries under a polynomial bound on the matrix norm and a global small-ball probability condition.
	
	\begin{theorem}\label{main-2}
		There exists an absolute constant $c_{1.6}>0$ such that the following statement holds.
		
		Let $M$ be an $m\times m$ random matrix with jointly Gaussian entries. Assume the global small-ball probability bound $P_{\alpha,T}(M)\ge \frac{1}{2}$, where $\alpha\ge 0$ and $T$ is the matrix type (real, complex, real symmetric or self-adjoint) of $M$. Then
		$$\mathbb{P}\left(\sigma_m(M)\le \frac{\varepsilon^2 \alpha^2}{ m^3 s}\right)\le c_{1.6} \varepsilon$$
		for every $\varepsilon\ge 0$ and $s>0$ with $\mathbb{P}\left(\sigma_1(M)> s\right)\le \frac{1}{8}.$
	\end{theorem}

	The condition $\mathbb{P}\left(\sigma_1(M)> s\right)\le \frac{1}{8}$ is just one of many equivalent ways to impose a polynomial bound on the matrix norm. 
	
	We provide an example to illustrate what may happen without the global small-ball probability condition. Let $M$ be an $m\times m$ matrix with entries $$a_{i,j}=\begin{cases}
		g&\mbox{, if }i+j=m+1,\\
		m&\mbox{, if }i+j=m+2,\\
		0&\mbox{, otherwise,}
	\end{cases}$$
	where $g$ is a real standard normal random variable. Then $M$ is a real symmetric matrix with jointly Gaussian entries. We even have a global small-ball probability bound for $M^T M$: $$\inf_{x\in S^{m-1}}\mathbb{P}\left(\norm{Mx}_2>\frac{1}{2}\right)\ge \frac{1}{2}.$$
	However, the global small-ball probability condition for $M$ $$\inf_{x\in S^{m-1}}\mathbb{P}\left(\left|x^* M x\right| >\alpha\right)\ge \frac{1}{2}$$ is not satisfied for any $\alpha\ge 0$. And with overwhelming probability, the minimum singular value of $M$ is exponentially small.

	\subsection{Matrix anti-concentration inequality with Gaussian coefficients}
	Spielman and Teng \cite{Spielman} introduced the smoothed analysis of an algorithm to explain why the simplex algorithm works well in practice. They assumed that the input matrix in real life is perturbed by a random matrix with i.i.d. Gaussian entries, which makes it well-conditioned. Specifically, for an $m\times m$ deterministic matrix $A$ and an $m\times m$ random matrix $M$ with independent standard normal entries, Sankar, Spielman and Teng \cite{Sankar} proved that $$\mathbb{P}\left(\sigma_m\left(A+M\right)\le \varepsilon\right)\le c \varepsilon \sqrt{m}$$ for an absolute constant $c>0$ and every $\varepsilon\ge 0$. Under the additional bound on the minimum singular value, many algorithms and heuristics enjoy much better smoothed complexity compared with their poor worst-case behaviors.
	
	In the past fifteen years, much effort has been put into the development of tools in random matrix theory under the smoothed analysis setting. For matrices with i.i.d. entries, Tao and Vu obtained the following result.
	\begin{theorem}\label{tao-vu}\cite{Tao}
		For any $\gamma\ge \frac{1}{2}$, $A\ge 0$ and any centered random variable $x$ with bounded second moment, there exists $c>0$ such that the following statement holds.
		
		Let $A$ be an $m\times m$ deterministic matrix with $\sigma_1\left(A\right)\le m^{\gamma}$. Let $M$ be an $m\times m$ random matrix whose entries are independent copies of $x$. Then we have
		$$\mathbb{P}\left(\sigma_m\left(A+M\right)\le m^{-(2A+1)\gamma}\right)\le c\left(m^{-A+o(1)}+\mathbb{P}\left(\sigma_1\left(M\right)\ge m^{\gamma}\right)\right).$$
	\end{theorem} 
	Better bounds are derived under stronger conditions such as $x$ is subgaussian \cite{Dong} and that $A$ has $\Omega(n)$ singular values which are $O(n)$ \cite{Jain}. The smoothed analysis of matrices with independent rows was conducted \cite{Tikhomirov2} by Tikhomirov. Farrell and Vershynin \cite{Farrell} studied the smoothed analysis of symmetric random matrices.
	
	From the perspective of theoretical computer science, these assumptions on independence are still way too strong while analyzing a randomized algorithm in the worst case scenario. For our application in sparse linear system solving, we need to develop a better minimum singular value bound for matrices without much entry-wise independence condition. First, we explore the possibilities of generalizing Theorem \ref{main-2}. 
	
	The proof of Theorem \ref{main-2} relies critically on the fact that the sum of independent normal variables is also normally distributed. Let $M_0, M_1,\ldots, M_n$ be $m\times m$ deterministic complex matrices, and $x_1,\ldots, x_n$ be independent copies of a random variable $x$. The general problem of bounding the minimum singular value of the matrix
	$$M:=M_0+\sum_{i=1}^n x_i M_i$$
	remains unsolved, even when $x$ follows a uniform distribution on the interval $[-1,1]$. 
	
	Another direction for generalization is to allow randomness in each coefficient matrix $M_i$. Note that the global small-ball probability bound on the random matrix $M$ does not even imply a nonzero second largest singular value $\sigma_{2}(M)$. In order to guarantee a nonzero minimum singular value $\sigma_{m}(M)$, we define $M$ as a sum of at least $m$ independent random matrices which jointly satisfy a global small-ball probability bound. In this spirit, we propose the following matrix anti-concentration inequality.

\begin{theorem}\label{Matrix_anticoncentration}
	There exists an absolute constant $c_{1.8}>0$ such that the following statement holds.
	
	Let $M_0,M_1,\ldots, M_n$ be independent complex $m\times m$ random matrices, and let $g_1,\ldots, g_n$ be independent real standard normal random variables. Let $M'$ be the random matrix uniformly chosen from $M_1,\ldots, M_n$.
	
	Assume the global small-ball probability bound $P_{\alpha,T}(M')\ge \beta$, where $\alpha \ge 0, 0<\beta \le 1$ and $T$ is a common matrix type of $M_0,M_1,\ldots, M_n$. Suppose that 
	$$n\ge \frac{c_{1.8}m}{\beta} \log \frac{2}{\beta}.$$
	Then the random matrix $$M:=M_0+\frac{1}{\sqrt{n}}\sum_{i=1}^n g_i M_i$$ satisfies
	$$\mathbb{P}\left(\sigma_m(M)\le \frac{\varepsilon^2 \alpha^2 \beta}{m^3 s}\right)\le c_{1.8}\left(\varepsilon+\mathbb{P}\left(\sigma_1(M)>s\right)\right)+ \exp\left(-\frac{n\beta}{c_{1.8}}\right)$$
	for each $\varepsilon\ge 0$ and $s\ge 0$.
\end{theorem}

One can find similarity among Theorem \ref{main-1}, Theorem \ref{tao-vu} and Theorem \ref{Matrix_anticoncentration}. In particular, Theorem \ref{Matrix_anticoncentration} bounds the minimum singular value of the matrix $\sum_{i=1}^n x_i M_i$ where $M_0, M_1,\ldots, M_n$ be deterministic complex matrices, and $x_1,\ldots, x_n$ be i.i.d. sparse Gaussian variables. This fact will be applied in the next subsection.

\subsection{Minimum singular value of Krylov space matrix}

Our main application of Theorem \ref{Matrix_anticoncentration} is to improve the algorithm for solving sparse linear systems. Let $A$ be an $n\times n$ real matrix with $nnz(A)$ nonzero entries and condition number $n^{O(1)}$. Peng and Vempala developed an algorithm \cite{PV21} that solves the linear system $Ax=b$ to accuracy $n^{-O(1)}$ in $\tilde{O}\left( \mathrm{nnz}(A)^{\frac{\omega-2}{\omega-1}}n^2+ n^{\frac{5\omega-4}{\omega+1}}\right)$ time, where $\omega<2.37286$ is the matrix multiplication exponent. When $A$ is sufficiently sparse, it is faster than the matrix multiplication.

This algorithm is a numerical version of those \cite{E06,E07} by Eberly, Giesbrecht, Giorgi, Storjohann and Villard, which deal with the finite field case and run in $\tilde{O}\left( \mathrm{nnz}(A)^{\frac{\omega-2}{\omega-1}}n^2\right)$ time. It is worth noting that Casacuberta and Kyng \cite{Casacuberta} further reduced the complexity in the finite field case by taking a different approach.

To achieve the stability, Peng and Vempala estimated the minimum singular value of the randomized Krylov space matrix as follows. 

	\begin{theorem}\label{peng-vem}\cite[Theorem 3.7]{PV}
	Let $A$ be an $n\times n$ real symmetric positive definite matrix with magnitudes of entries at most $n^{-1}$, and $\alpha<n^{-10}$ be a positive real number such that
	\begin{enumerate}
		\item all eigenvalues of $A$ are at least $\alpha$, and
		\item all pairs of eigenvalues of $A$ are separated by at least $\alpha$.
	\end{enumerate}
	Let $s$ and $m$ be positive integers with $n^{0.01}\le m\le n^{\frac{1}{4}}$ and $sm\le n-5m$. Let $G$ denote an $n\times s$ sparse Gaussian matrix where each entry is set to $\mathcal{N}(0, 1)$ with probability $\frac{h}{n}$, and $0$ otherwise. Suppose that $$h\ge 10000 m^3 \log\frac{1}{\alpha}.$$
	Then the Krylov space matrix
	$$K:=\left[
	\begin{array}{c|c|c|c|c}
		G & A G & A^2 G &  \ldots & A^{m-1} G
	\end{array}
	\right]
	$$
	satisfies  $$\mathbb{P}\left(\sigma_n(K)\le \alpha^{5m}\right)\le n^{-2}.$$
	\end{theorem}

	Our final main result is to provide a better lower bound on the minimum singular value of the Krylov space matrix in the following way.

	\begin{theorem}\label{main-4}
		There exists an absolute constant $c_{1.10}>0$ such that the following statement holds.
		
		Let $m$ be a positive factor of $n$ with $$n>c_{1.10} m^2\log n +1.$$ Let $A_1,\ldots, A_m$ be $n\times n$ real deterministic matrices with $\sigma_1\left(A_i\right)\le 1$ for each $i=1,\ldots, m$.
		Suppose that there exists a positive real number $\alpha\le n^{-\log n}$ such that
			$$\sigma_{n-k+1}\left(\sum_{i=1}^k x_i A_i\right)\ge \alpha^{k}$$
		for every $k\in\left\{1,\ldots, m\right\}$ and every unit vector $(x_1,\ldots,x_k)$ in $S^{k-1}$.
		
		Let $G$ denote an $n\times \frac{n}{m}$ sparse Gaussian matrix where each entry is set to $\mathcal{N}(0, 1)$ with probability $\frac{h}{n}$, and $0$ otherwise. Suppose that $$h\ge c_{1.10} m \log n \log\frac{1}{\alpha}.$$
		Then the Krylov space matrix
		$$K:=\left[
		\begin{array}{c|c|c|c}
			A_1 G & A_2 G &\ldots & A_m G
		\end{array}
		\right]
		$$
		satisfies  $$\mathbb{P}\left(\sigma_n(K)\le \alpha^{500m}\right)\le\alpha.$$
	\end{theorem}

	There are two main differences between Theorem \ref{peng-vem} and Theorem \ref{main-4}. First, the additional time cost of Peng and Vempala's algorithm comes from the $m^3$ factor in the sparsity bound. With our Theorem \ref{main-4}, the overall time complexity can be reduced to $\tilde{O}\left( \mathrm{nnz}(A)^{\frac{\omega-2}{\omega-1}}n^2\right),$ as the complexity derived in \cite{E06,E07} in the finite field case. Second, with the condition $sm\le n-5m$, the minimum singular value bound derived in \ref{peng-vem} is for rectangular matrices. Therefore, they need an extra padding step in their algorithm to form a square matrix. With Theorem  \ref{main-4}, we remove the padding step to simplify the algorithm, although this does not give asymptotic speedups. Specifically speaking, we replace the block Krylov space algorithm shown in \cite[Figure 2]{PV} with the following simplified version.
	\begin{figure}[h]
		\fbox{\parbox{\textwidth}{
				{\bf \textsc{SimplifiedBlockKrylov}(
					$\textsc{MatVec}_{A}( x, \delta)$:
					symmetric matrix given as implicit matrix vector muliplication access,
					$\alpha_{A}$: eigenvalue range/separation bounds for $A$,
					$m$: Krylov step count which is a factor of $n$
					)} 
				
				\begin{enumerate}
					
					\item (FORM KRYLOV SPACE) 
					\begin{enumerate}
						\item Set $s \leftarrow \frac{n}{m}$,
						$h \leftarrow O\left(m\log n \log\frac{1}{ \alpha_{A}} \right)$.
						Let $G$ be an $n \times s$ random matrix with each entry independently
						set to $\mathcal{N}(0, 1)$
						with probability $\frac{h}{n}$, and $0$ otherwise.
						
						\item Implicitly compute the block Krylov space
						\[
						K
						=
						\left[
						\begin{array}{c|c|c|c|c}
							G & A G & A^{2} G & \ldots & A^{m - 1} G
						\end{array}
						\right].
						\]
					\end{enumerate}
					
					\item (SPARSE INVERSE)
					Use fast solvers for block Hankel matrices to
					obtain a solver for the matrix $\left( A K\right)^{T}  AK$ and in turn a solve to arbitrary error which we denote $\textsc{Solve}_{\left( A K\right)^{T} AK}(\cdot, \varepsilon)$.
					
					\item (SOLVE and UNRAVEL) Return the operator
					$$K \cdot 
					\textsc{Solve}_{(AK)^{T}AK} \left( (AK)^{T}x, \alpha_{A}^{O(m)} \right)
					$$
					as an approximate solver for $A$.
				\end{enumerate}
		}}
		\caption{Pseudocode for simplified block Krylov space algorithm.
		}
		\label{fig:solver}
	\end{figure}

	The subsequent work of sparse linear regression by Ghadiri, Peng and Vempala \cite{Ghadiri} also benefit from our improvement because they used the algorithm in \cite{PV21} as a subroutine. 
	
	Let us also discuss other assumptions in Theorem \ref{main-4}. We will take $m=\Theta\left( \mathrm{nnz}\left(A\right)^{-\frac{1}{\omega-1}}n\right)$ in the algorithm, so the assumption $n>c_{1.10} m^2\log n +1$ can be satisfied. The assumption that $m$ is a factor of $n$ can be satisfied by adding trivial rows and columns to the matrix $A$. The assumption $$\sigma_{n-k+1}\left(\sum_{i=1}^k x_i A_i\right)\ge \alpha^{k}$$ can be deduced from the property \cite[Lemma 5.1]{PV} of Vandermonde matrices when $A_i=A^{i-1}$ and $A$ satisfies the assumption of Theorem \ref{peng-vem}. 
	
	The proof of Theorem \ref{main-4} is a hybrid of the classical tools developed in the non-asymptotic theory for matrices with i.i.d. entries and the matrix anti-concentration inequality. We decompose the unit sphere as classes of ``compressible'' and ``incompressible'' vectors in a way consistent with our global small-ball condition. Then our matrix anti-concentration inequality can be viewed as a matrix and Gaussian version\footnote{As a parallel result, a vector version of the Littlewood-Offord theorem was established \cite{Tao2} by Tao and Vu using Esséen's concentration inequality \cite{Esseen}.} of the Littlewood-Offord theorem. Therefore the inverse Littlewood-Offord method introduced in \cite{Tao3} applies to our case. 
	
	A technical difficulty of the proof comes from the compressible side. The single vector probability is not small enough to compensate the union bound, so we need to group these vectors in subspaces to utilize the matrix anti-concentration inequality as on the incompressible side. However, it is possible that a subspace contains vectors with different data compression ratio, which jeopardizes the probability bound. To overcome this difficulty, we partition the matrix $G$ into around $\log m$ parts, and handle each level of data compression ratio inductively.

	\section{Proof of Theorem \ref{main-2}}
	
	We first prove Theorem \ref{intro-non-singular-1}, Theorem \ref{intro-non-singular-2} and their complex counterparts in the following unified way.
	
	\begin{theorem}\label{o-version-s}
		Let $S$ be a real affine space of $m\times m$ complex matrices. Let $T$ be a common type of the matrices in $S$. Suppose that there exists $M\in S$ with $x^* M y\neq 0$ for every pair of $m$-dimensional unit column vectors $x , y$ with the following restrictions:
		\begin{enumerate}[label=(\arabic*)]
			\item if $T\in\{\mbox{real}, \mbox{real symmetric}\}$, let $x$ and $y$ be real;
			\item if $T\in\{\mbox{real symmetric}, \mbox{self-adjoint}\}$, let $x=y$.
		\end{enumerate}
		Then there exists $M\in S$ with $\det M\neq 0$.
	\end{theorem}
	\begin{proof}
		Let $M$ be a matrix in $S$ with maximum rank $k$. Then there exist $m\times k$ matrices $U$ and $V$ with orthonormal columns, such that $U^* M V$ is non-singular. 
		
		Assume $k<m$ for the sake of contradiction. Then there exist unit column vectors $x, y$ with $U^* x= V^* y=0$. By the spectral theorem, we assume that $U, V, x,y$ are real if $M$ is real, and assume that $U=V$ and $x=y$ if $M$ is self-adjoint. By the condition we imposed on $S$, there exists $M'\in S$ with $x^* M' y\neq 0$. 
		
		Let $U'$ (resp. $V'$) denote the $m\times (k+1)$ matrices obtained by appending $x$ (resp. $y$) to $U$ (resp. $V$). Then the leading coefficient of the polynomial $\det\left(U'^* (t M - M') V'\right)$ with respect to $t$ is $-\left(x^* M' y\right)\det\left(U^* M V\right)$, which is non-zero. Hence there exists a real number $t\neq 1$ such that $\det \left(U'^* (t M - M') V'\right)\neq 0$, so the rank of $\frac{1}{t-1}(t M - M')$ is at least $k+1$, which is contradictory to our choice of $M$.
		
		Therefore, we have $\det M\neq 0$.
	\end{proof}

	We define a partial determinant as follows.
	
	\begin{definition}\label{def-det}
	For an $m\times m$ matrix $A$ and an integer $k$ with $0\le k\le m$, we define a partial determinant
	$$D_{k}(A):=\prod_{i=1}^{k}\sigma_i(A)$$
	to be the product of $k$ greatest singular values of $A$.	
	\end{definition}
	
	By Cauchy interlacing theorem, we have $$\sigma_k(A)=\max_{U,V} \sigma_k\left(U^* A V\right)$$
	and
	$$D_k(A)=\max_{U,V} \left|\det\left(U^* A V\right)\right|$$
	where $U,V$ range over $m\times k$ matrices with orthonormal columns. Additionally, if $A$ is real, we can achieve the maximum when $U, V$ are real; if $A$ is self-adjoint, we can assume $U=V$ by the spectral theorem; if $A$ is real symmetric, we can combine the above two conditions.
	
	The following lemma allows us to make computation with $\mathbb{E}\left(\log D_k (M)\right)$ for each $k=0,1,\ldots, m$. It implies the special case of Theorem \ref{main-2} where $\alpha=\varepsilon=0$.
	\begin{lemma}\label{lemma-finite}	
		Let $M$ be an $m\times m$ random matrix with jointly Gaussian entries. Assume the global small-ball probability bound $P_{0,T}(M)>0$, where $T$ is the matrix type of $M$. Then 
		$$\mathbb{E}\left(\left|\log D_k(M)\right|\right)<\infty$$
		for each $k=0,1,\ldots, m$.
	\end{lemma}
	\begin{proof}
		Because the greatest singular value $\sigma_1(M)$ of $M$ is sub-Gaussian, we have
		\begin{align*}
			&\mathbb{E}\left(\max\left \{ \log D_{k}(M) ,0\right\}\right)\\
			= &\mathbb{E}\left(\max\left \{ \sum_{i=1}^{k}\log \sigma_{i}(M) ,0\right\}\right)\\
			\le& k\cdot \mathbb{E}\left(\max\left \{ \log \sigma_{1}(M) ,0\right\}\right) \\
			=&k \int_{0}^{\infty} \mathbb{P}\left( \log  \sigma_{1}(M) \ge t\right)dt\\
			<&\infty.
		\end{align*}
		
		Because $M$ has jointly Gaussian entries, we assume $$M=M_0+ \sum_{i=1}^{m^2} g_i M_i$$
		where $M_0, M_1,\ldots, M_{m^2}$ are  deterministic complex matrices and $g_1,\ldots, g_{m^2}$ are real independent standard normal random variables. Let $S$ be the support of the random matrix $M$. Then $S$ satisfies the condition of Theorem \ref{o-version-s}, so there exists $M'\in S$ with $\det M'\neq 0$. Suppose that $$M'=M_0+\sum_{i=1}^{m^2} t_i M_i$$ where $t_1,\ldots, t_{m^2}$ are real numbers.
		
		We will prove 
		\begin{equation}\label{ind-step}
			\mathbb{E}\left(\max\left\{ -\log D_k\left(M'+\sum_{i=1}^l (g_i-t_i) M_i\right) ,0\right\}\right)<\infty
		\end{equation}
		for each $l=0,1,\ldots, m^2$ by induction on $l$. The base case $l=0$ is trivial because $M'$ is deterministic with $\det M'\neq 0$.
		
		Assume (\ref{ind-step}) holds for $l-1$. Let $$M''=M'+\sum_{i=1}^{l-1} (g_i-t_i) M_i.$$ Let $U,V$ be $m\times k$ matrices with orthonormal columns such that $D_{k}(M'')= \left|\det\left(U^* M'' V\right)\right|.$ Then, by inductive hypothesis, we have $$\mathbb{E}\left( \max \left\{ -\log D_k(M''),0\right\}\right)<\infty,$$ which implies $$\mathbb{P}\left( D_k(M'')=0\right)=0.$$ Let $\lambda_1,\ldots, \lambda_{k}$ be the eigenvalues of $\left( U^* M_l V\right) \left(U^* M'' V\right)^{-1}$. Then we have
		\begin{align*}
			&\mathbb{E}\left(\max\left\{ -\log D_k\left(M''+(g_l-t_l)M_l \right) ,0\right\}\right)\\
			\le &\mathbb{E}\left(\max\left\{ -\log \det \left|U^* M'' V+(g_l-t_l) U^* M_l V \right| ,0\right\}\right)\\
			= &\mathbb{E}\left( \max \left\{ -\log\det \left|U^* M'' V \right|\right.\right.\\
			&\left.\left.-\log \det \left|I+(g_l-t_l) \left(U^* M_l V\right)\left(U^* M'' V\right)^{-1} \right| ,0  \right\} \right)\\
			= &\mathbb{E}\left( \max \left\{ -\log D_k(M'')-\sum_{i=1}^l\log\left| 1+(g_l-t_l)\lambda_i\right| ,0 \right\}\right)\\
			\le &\mathbb{E}\left( \max \left\{ -\log D_k(M''),0\right\}\right) +\sum_{i=1}^l\mathbb{E}\left( \max \left\{-\log\left| 1+(g_l-t_l)\lambda_i\right| ,0 \right\}\right)\\
			\le &\mathbb{E}\left( \max \left\{ -\log D_k(M''),0\right\}\right)+ l\cdot \sup_{\lambda\in \mathbb{C}} \mathbb{E}\left( \max \left\{-\log\left| 1+(g_l-t_l)\lambda\right| ,0 \right\}\right).
		\end{align*}
		We have
		\begin{align*}
			&\sup_{\lambda\in \mathbb{C}}\mathbb{E}\left( \max \left\{-\log\left| 1+(g_l-t_l)\lambda \right| ,0 \right\}\right)\\
			\le &\sup_{\lambda\in \mathbb{C}}\mathbb{E}\left( \max \left\{-\log\left|\Re\left( 1+(g_l-t_l)\lambda \right)\right| ,0 \right\}\right)\\
			= &\sup_{\lambda\in \mathbb{C}}\mathbb{E}\left( \max \left\{-\log\left| 1+(g_l-t_l)\Re(\lambda) \right| ,0 \right\}\right)\\
			= &\sup_{\lambda\in \mathbb{R}}\mathbb{E}\left( \max \left\{-\log\left| 1+(g_l-t_l)\lambda\right|,0 \right\}\right)\\
			=&\sup_{\lambda\in \mathbb{R}} \int_{0}^{\infty} \mathbb{P}\left( \left| 1+(g_l-t_l)\lambda\right| \le \exp(-t)\right)dt\\
			=&\max\left\{  \sup_{\lambda\in \mathbb{R},|t_l \lambda|>\frac{1}{4}} \int_{0}^{\infty} \mathbb{P}\left( \left| 1-t_l \lambda+ g_l\lambda \right| \le \exp(-t)\right)dt,\right.\\
			&\left. 1+\sup_{\lambda\in \mathbb{R},|t_l \lambda|\le \frac{1}{4}} \int_{1}^{\infty} \mathbb{P}\left( \left| 1-t_l \lambda+ g_l\lambda\right| \le \exp(-t)\right)dt \right\}\\
			\le& \max\left\{  \sup_{\lambda\in \mathbb{R},|t_l \lambda|>\frac{1}{4}} \int_{0}^{\infty} \frac{\exp(-t)}{ \sqrt{2\pi} |\lambda|} dt,\right.\\
			&\left.1+ \sup_{\lambda\in \mathbb{R},|t_l \lambda|\le \frac{1}{4}} \int_{1}^{\infty}   \frac{\exp(-t)}{ \sqrt{2\pi} |\lambda|}\exp\left(-\frac{1}{32\lambda^2}\right) dt \right\}\\
			\le& \max\left\{  2\sqrt{\frac{2}{\pi}}|t_l|,1+\frac{1}{\sqrt{2\pi}e}\sup_{\lambda\in \mathbb{R}}\frac{1}{ |\lambda|}\exp\left(-\frac{1}{32\lambda^2}\right)\right\}\\
			\le& \max\left\{ 2\sqrt{\frac{2}{\pi}}|t_l|,1+\frac{1}{\sqrt{2\pi}e}\sup_{\lambda\in \mathbb{R}}\frac{1}{ |\lambda|}\left(1+\frac{1}{32\lambda^2}\right)^{-1}\right\}\\
			\le& \max\left\{ 2\sqrt{\frac{2}{\pi}}|t_l|,1+\frac{2}{\sqrt{\pi}e}\right\}\\
			<&\infty.
		\end{align*}
		Thus we have $$\mathbb{E}\left(\max\left\{ -\log D_k\left(M''+(g_l-t_l)M_l \right) ,0\right\}\right)<\infty.$$ 
		
		By the principle of induction, we have
		\begin{align*}
			&\mathbb{E}\left(\max\left\{ -\log D_k(M) ,0\right\}\right)
	\\
	=&\mathbb{E}\left(\max\left\{ -\log D_k\left(M'+\sum_{i=1}^{m^2} (g_i-t_i) M_i\right) ,0\right\}\right)\\
	<&\infty.
		\end{align*}
		Therefore, we have 
		\begin{align*}&\mathbb{E}\left(\left|\log D_{k}(M)\right|\right)\\
			=&\mathbb{E}\left(\max\left \{ \log D_{k}(M) ,0\right\}\right)+\mathbb{E}\left(\max\left \{ -\log D_{k}(M) ,0\right\}\right)\\
			<&\infty.\end{align*}\end{proof}
	
	Let $u_0$ be a positive real number to be determined later. Let $u$ be a random variable with probability density function
	\begin{equation}\label{defu}
		f_u(t)=\begin{cases}
			\frac{1}{2} u_0 t^{-\frac{3}{2}}&, \mbox{ if } t> u_0^2,\\
			0&, \mbox{ if } t< u_0^2.
		\end{cases}
	\end{equation}
	
	The following inequality is the key ingredient in the proof of Theorem \ref{main-2}.
	
	\begin{lemma}\label{lemma-monotone-scalar}
		Let $u$ be the random variable defined by (\ref{defu}). For every complex number $\lambda$ we have
		$$\mathbb{E} \left(\log |1+u \lambda|\right)\ge 0.$$
	\end{lemma}
	\begin{proof}
		If $\Re\left(\lambda\right)\ge 0$ or $\Re\left(\lambda\right)\le -2u_0^{-2}$, then we have
		\begin{align*}
			&\mathbb{E} \left(\log \left|1+ u\lambda\right|\right)\\
			\ge &\mathbb{E} \left(\log \left|1+ u\Re\left(\lambda\right)\right|\right)\\
			\ge& 0.
		\end{align*}
		
		If $-2u_0^{-2}<\Re\left(\lambda\right)< 0$, then we have
		\begin{align*}	
			&\mathbb{E} \left(\log \left|1+ u\lambda\right|\right)\\
			\ge &\mathbb{E} \left(\log \left|1+ u\Re\left(\lambda\right)\right|\right)\\
			=&\frac{1}{2} u_0\int_{u_0^2}^{\infty} t^{-\frac{3}{2}} \log\left| 1+  t \Re\left(\lambda\right)  \right|dt \\
			=&\frac{1}{2} u_0 \sqrt{- \Re\left(\lambda\right)} \int_{-u_0^2\Re\left(\lambda\right)}^{\infty} t^{-\frac{3}{2}} \log\left| 1- t \right|dt\\
			=&-\frac{1}{2} u_0\sqrt{- \Re\left(\lambda\right)} \int_0^{-u_0^2\Re\left(\lambda\right)} t^{-\frac{3}{2}} \log\left| 1- t \right|dt\\
			\ge &0.
		\end{align*}
	\end{proof}
	
	We introduce a matrix version of Lemma \ref{lemma-monotone-scalar}.
	
	\begin{lemma}\label{lemma-monotone}
		Let $u$ be the random variable defined by (\ref{defu}). Assume that $0\le k\le m$. Let $A_0,A_1,\ldots, A_n$ be $m\times m$ deterministic complex matrices with $D_{k}(A_0)\neq 0$. Then we have 
		$$\mathbb{E}\left(\log  D_{k} \left(\sum_{i=0}^n u^i A_i\right) \right)\ge \log D_{k}(A_0).$$
	\end{lemma}
	\begin{proof}
		Let $U,V$ be $m\times k$ matrices with orthonormal columns such that $D_{k}(A_0)= \left|\det\left(U^* A_0 V\right)\right|.$ Let $-\frac{1}{\lambda_1},\ldots,-\frac{1}{\lambda_{kn}}$ be the roots of the polynomial $$f(x):=\det\left(I+ \sum_{i=1}^n x^i\left( U^* A_i V\right) \left(U^* A_0 V\right)^{-1}\right).$$
		By Lemma \ref{lemma-monotone-scalar}, we have
		\begin{align*}
			&\mathbb{E}\left(\log  D_{k} \left(\sum_{i=0}^n u^i A_i\right) \right)- \log D_{k}(A)\\
			=&\mathbb{E}\left(\log  D_{k} \left(\sum_{i=0}^n u^i A_i\right) \right)-\log \left|\det\left(U^* A_0 V\right)\right|\\
			\ge&\mathbb{E}\left(\log\left|\det\left(U^* \left(\sum_{i=0}^n u^i A_i\right) V\right)\right| \right)- \log\left|\det\left(U^* A_0 V\right)\right|\\
			=&\mathbb{E}\left(\log\left|\det\left(I+ \sum_{i=1}^n u^i\left( U^* A_i V\right) \left(U^* A_0 V\right)^{-1}\right)\right| \right)\\
			=&\sum_{i=1}^{k n}\mathbb{E}\left( \log\left|1+u \lambda_i\right|\right)\\
			\ge &0.
		\end{align*}
	\end{proof}
	
	The following lemma provides another lower bound of the increment in the expectation, given that a singular value of $A_0$ is small.
	
	\begin{lemma}\label{lemma-positive}
		 Let $A_0$ and $A_2$ be two $m\times m$ deterministic matrices $A_0, A_2$ with $$0<\sigma_{m}(A_0)\le \frac{u_0^2\alpha^2}{2m s}.$$ Let $M$ be an $m\times m$ random matrix. Assume the global small-ball probability bound $P_{2\alpha,T}(M)\ge \frac{1}{4}$, where $T$ is the matrix type of $A_0$. Let $u$ be a random variable defined by (\ref{defu}), independent with $M$.
		 Then we have 
		\begin{align*}
			&\mathbb{E}\left(\log  D_{m} \left(A_0+u M+u^2 A_2\right) +\log D_{m-1} \left(A_0+u M+u^2 A_2\right) 
			\right)\\
			\ge&\log D_{m}(A_0) + \log D_{m-1}(A_0)  +\frac{1}{2}-2\cdot\mathbb{P}\left(\sigma_1(M)\ge s\right),
		\end{align*}
	given that the expectation exists.
	\end{lemma}	
	\begin{proof}
		
		Let $x_1, \ldots, x_m$ and $y_1,\ldots, y_m$ be the the left-singular vectors and right-singular vectors of $A_0$ with $A_0 y_i= \sigma_i(A_0) x_i$ for $i=1,\ldots,m$. If $A_0$ is real, we additionally assume that $x_1, \ldots, x_m$ and $y_1,\ldots, y_m$ are real; If $A_0$ is self-adjoint, we additionally assume that $x_i=y_i$ or $x_i=-y_i$ for each $i=1,\ldots,m$. 
		
		Let $E_0$ denote the event that $\sigma_1(M)< s,$ $\left| x_m^* M y_m\right|\ge 2\alpha$ and $$\Re\left(\frac{\sigma_{m}(A_0)x_k^* M y_k}{\sigma_{k}(A_0) x_m^* M y_m}\right)\ge -\frac{1}{2m}$$ for each $k=1,\ldots,m-1$. When $E_0$ occurs, we have
		\begin{align*}
			 &\left|\Tr{M A_0^{-1} }\right|\\
			=&\left|\sum_{i=1}^m  x_i^* M A_0^{-1} x_i\right|\\
			=&\frac{\left|x_m^* M y_m\right|}{\sigma_m(A_0)} \left|1+\sum_{i=1}^{m-1} \frac{\sigma_{m}(A_0)x_i^* M y_i}{\sigma_{i}(A_0) x_m^* M y_m}\right|\\
			\ge & \frac{\left|x_m^* M y_m\right|}{\sigma_m(A_0)} \Re \left(1+\sum_{i=1}^{m-1} \frac{\sigma_{m}(A_0)x_i^* M y_i}{\sigma_{i}(A_0) x_m^* M y_m}\right)\\
			\ge &\frac{\left|x_m^* M y_m\right|}{2\sigma_m(A_0)}\\
			= &\frac{\left|x_m^* M y_m\right|^2}{2\sigma_m(A_0)\left|x_m^* M y_m\right|}\\
			\ge &\frac{2\alpha^2}{s\sigma_m(A_0)}\\
			\ge &\frac{4m}{u_0^2}.
		\end{align*}
	
		Let $-\frac{1}{\lambda_1},\ldots, -\frac{1}{\lambda_{2m}}$ be the roots of the polynomial $$f_1(x):=\det\left(I+ x M A_0^{-1}+x^2 A_2 A_0^{-1} \right).$$ When $E_0$ occurs, we have
		\begin{align*}
			&\max_{1\le i\le 2m}|\lambda_i|\\
			\ge&\frac{1}{2m}\left|\sum_{i=1}^{2m}\lambda_i\right|\\
			 = & \frac{1}{2m}\left|\Tr{MA_0^{-1}}\right|\\
			\ge& 2u_0^{-2}.
		\end{align*}
		By Lemma \ref{lemma-monotone-scalar}, when $E_0$ occurs, we have
		\begin{align*}
			&\mathbb{E}\left(\log  D_{m} \left(A_0+u M+u^2 A_2\right)\;\middle|\; M\right) - \log D_{m}(A_0)\\
			=&\mathbb{E}\left( \log\left|\det\left(A_0+u M+u^2 A_2\right) \right|\;\middle|\; M\right)-\log\left|\det A_0 \right|\\
			=&\mathbb{E}\left(\log\left|\det\left(I+u M A_0^{-1} +u^2 A_2 A_0^{-1}\right)\right|\;\middle|\; M\right)\\
			=&\sum_{i=1}^{2m} \mathbb{E}\left(\log|1+u \lambda_i |\;\middle|\; M\right)\\
			\ge &\inf_{|\lambda|\ge 2u_0^{-2}} \mathbb{E}\left(\log|1+u \lambda | \right)\\
			\ge &\inf_{|\lambda|\ge 2u_0^{-2}} \mathbb{E}\left(\log\frac{|u \lambda|}{2} \right)\\
			= &2.
		\end{align*}
		
		For $k=1,\ldots, m-1$, let $E_k$ denote the event that $\sigma_1(M)< s,$ $\left|x_m^* M y_m\right|\ge 2\alpha$ and $$\Re\left(\frac{\sigma_{m}(A_0)x_k^* M y_k}{\sigma_{k}(A_0) x_m^* M y_m}\right)< -\frac{1}{2m}.$$
		For $\varepsilon_0=\pm 1$, let $U_{\varepsilon_0}$ denote the $m\times (m-1)$ matrix with columns $x_1,\ldots, x_{k-1}$, $x_{k+1},\ldots, x_{m-1}$ and $$\frac{x_k+\varepsilon_0 x_m}{\sqrt{2}},$$ and let $V_{\varepsilon_0}$ denote the $m\times (m-1)$ matrix with columns $y_1,\ldots, y_{k-1}$, $y_{k+1},\ldots, y_{m-1}$ and $$\frac{\sigma_m(A_0)y_k -\varepsilon_0\sigma_k(A_0) y_m }{\sqrt{\sigma_m(A_0)^2 +\sigma_k(A_0)^2}}.$$
		Then $U_{\varepsilon_0}$ and $V_{\varepsilon_0}$ have orthonormal columns, and $U_{\varepsilon_0}^* A_0 V_{\varepsilon_0}$ is a diagonal matrix with the lower right entry being zero. The constant and linear terms of the polynomial $$f_2(x):=\frac{1}{2}\det\left(U_{1}^*\left(A_0+u M+u^2 A_2\right) V_{1}\right)+\frac{1}{2}\det\left(U_{-1}^*\left(A_0+u M+u^2 A_2\right) V_{-1}\right)$$ are zero and
		$$\frac{D_{m-1}(A_0) \left(\sigma_k(A_0)x_m^* M y_m-\sigma_m(A_0) x_k^* M y_k\right)} {\sigma_k(A_0)\sqrt{2\sigma_k(A_0)^2+2\sigma_m(A_0)^2}} x.$$
		When $E_k$ occurs, we have
		\begin{align*}
			&\frac{D_{m-1}(A_0) \left|\sigma_k(A_0)x_m^* M y_m-\sigma_m(A_0) x_k^* M y_k\right|} {\sigma_k(A_0)\sqrt{2\sigma_k(A_0)^2+2\sigma_m(A_0)^2}}\\
			=&\frac{D_{m-1}(A_0) \left|x_m^* M y_m
			\right|}{\sqrt{2\sigma_k(A_0)^2+2\sigma_m(A_0)^2}}\left|1-\frac{\sigma_{m}(A_0)x_k^* M y_k}{\sigma_{k}(A_0) x_m^* M y_m}\right|\\
			\ge& \frac{D_{m-1}(A_0) \left|x_m^* M y_m
				\right|}{\sqrt{2\sigma_k(A_0)^2+2\sigma_m(A_0)^2}}\Re\left(1-\frac{\sigma_{m}(A_0)x_k^* M y_k}{\sigma_{k}(A_0) x_m^* M y_m}\right)\\
			\ge& \frac{D_{m-1}(A_0) \left|x_m^* M y_m
				\right|}{ 2\sigma_k(A_0)}\\
			=&\frac{D_{m-1}(A_0) \left|x_m^* M y_m
				\right|^2}{ 2\sigma_m(A_0)\left|x_k^* M y_k\right|}\left|\frac{\sigma_{m}(A_0)x_k^* M y_k}{\sigma_{k}(A_0) x_m^* M y_m}\right|\\
			\ge&\frac{D_{m-1}(A_0) \left|x_m^* M y_m
				\right|^2}{ 2\sigma_m(A_0)\left|x_k^* M y_k\right|}\Re\left(-\frac{\sigma_{m}(A_0)x_k^* M y_k}{\sigma_{k}(A_0) x_m^* M y_m}\right)\\
			\ge&\frac{D_{m-1}(A_0) \left|x_m^* M y_m
				\right|^2}{ 4 m \sigma_m(A_0)\left|x_k^* M y_k\right|}\\
			\ge &\frac{D_{m-1}(A_0)\alpha^2}{ m s \sigma_m(A_0)}\\
			\ge &\frac{2 D_{m-1}(A_0)}{u_0^2}.
		\end{align*}
		Let $0,-\frac{1}{\lambda_1},\ldots, -\frac{1}{\lambda_{2m-1}}$ be the roots of the polynomial $f_2(x)$. Then by Lemma \ref{lemma-monotone-scalar}, when $E_k$ occurs, we have
		\begin{align*}
			&\mathbb{E}\left(\log  D_{m-1} \left(A_0+u M+u^2 A_2\right)\;\middle|\; M \right)\\
			\ge &\mathbb{E}\left(\max\left\{ \log\left|\det\left(U_1^*\left(A_0+u M+u^2 A_2\right) V_1\right)\right|,\right.\right.\\
			&\left.\left.\log\left|\det\left(U_{-1}^*\left(A_0+u M+u^2 A_2\right) V_{-1}\right)\right|\right\}\;\middle|\; M \right)\\
			\ge& \mathbb{E}\left( \log\left|\frac{1}{2}\det\left(U_{1}^*\left(A_0+u M+u^2 A_2\right) V_{1}\right)\right.\right.\\
			&\left.\left.+\frac{1}{2}\det\left(U_{-1}^*\left(A_0+u M+u^2 A_2\right) V_{-1}\right) \right|\;\middle|\; M \right)\\
			=&\mathbb{E}\left(\log\left(\frac{D_{m-1}(A) \left|\sigma_k(A)x_m^* M y_m-\sigma_m(A) x_k^* M y_k\right|} {\sigma_k(A)\sqrt{2\sigma_k(A)^2+2\sigma_m(A)^2}} \right)\;\middle|\; M \right)\\
			&+\mathbb{E}\left(\log u\right)+\sum_{i=1}^{2m-1} \mathbb{E}\left(\log|1+u \lambda_i | \;\middle|\; M\right)\\
			\ge &\log\left(\frac{2 D_{m-1}(A_0)}{u_0^2} \right) +\mathbb{E}\left(\log u\right)\\
			=&\log D_{m-1}(A_0)+ 2 + \log 2.
		\end{align*}
		
		Therefore, by Lemma \ref{lemma-monotone}, we have 
		\begin{align*}
			&\mathbb{E}\left(\log  D_{m} \left(A_0+u M+u^2 A_2\right) +\log D_{m-1} \left(A_0+u M+u^2 A_2\right)
			\right)\\
			\ge &\log D_{m}(A_0) + \log D_{m-1}(A_0)+ 2\cdot\mathbb{P}\left(\bigcup\limits_{i=0}^{m-1} E_i \right) \\
			= & \log D_{m}(A_0) + \log D_{m-1}(A_0)+ 2\cdot\mathbb{P}\left( \sigma_1(M)< s, \left| x_m^* M y_m\right|\ge 2\alpha\right)\\
			\ge&\log D_{m}(A_0) + \log D_{m-1}(A_0)+\frac{1}{2}-2\cdot\mathbb{P}\left( \sigma_1(M)\ge s\right).
		\end{align*}
	\end{proof}
	
	Now we can prove Theorem \ref{main-2} by Markov's inequality.
	
	\begin{proof}[Proof of Theorem \ref{main-2}]
		Let $M'$ be an independent copy of $M$. Because $M$ has jointly Gaussian entries, we have 
		$$(1+2t)M+(2t+2 t^2) M' -2t\cdot \mathbb{E}(M)\sim (1+2t+2t^2) M$$ 
		for every real number $t$. 
		
		By Definition \ref{def-gs}, we have
		\begin{align*}
			&P_{2\alpha,T}\left(2M+2M'-2\cdot \mathbb{E}(M)\right)\\
			=&\inf_{x,y}\mathbb{P} \left(\left|x^* \left(M+M'-\mathbb{E}(M)\right) y\right|> \alpha \right)\\
			\ge& \inf_{x,y}\left(\mathbb{P} \left(\left|x^* M' y\right|> \alpha \right) \mathbb{P}\left(\left|x^* \left(M+M'-\mathbb{E}(M)\right) y\right|\ge \left|x^* M' y\right| \;\middle|\;\left|x^* M' y\right|> \alpha \right)\right) \\
			\ge& P_{\alpha, T}\left(M'\right)\inf_{x,y} \mathbb{P}\left(\left|x^* \left(M+M'-\mathbb{E}(M)\right) y\right|\ge \left|x^* M' y\right| \;\middle|\;\left|x^* M' y\right|> \alpha \right) \\
			=& P_{\alpha, T}\left(M\right)\inf_{x,y} \mathbb{P}\left(\left|1+\frac{x^* \left(M-\mathbb{E}(M)\right) y}{\left|x^* M' y\right|}\right|\ge 1  \;\middle|\;\left|x^* M' y\right|> \alpha \right) \\
			\ge& P_{\alpha, T}\left(M\right)\inf_{x,y}  \mathbb{P}\left(\Re\left(x^* \left(M-\mathbb{E}(M)\right) y \right)\ge 0 \;\middle|\;\left|x^* M' y\right|> \alpha \right)\\
			\ge& \frac{1}{2}P_{\alpha, T}\left(M\right)\\
			\ge& \frac{1}{4},
		\end{align*}
		where $(x,y)$ ranges over $S_{m,T}$.
		
		By Lemma \ref{lemma-finite}, we have 
		\begin{align*}
			&\mathbb{E}\left(\log D_{k} \left(M+2u \left(M+M'-\mathbb{E}(M)\right)+2u^2 M' \right) - \log  D_{k}  (M) \right)\\
			=&\mathbb{E}\left(\mathbb{E}\left(\log D_{k} \left(M+2u \left(M+M'-\mathbb{E}(M)\right)+2u^2 M' \right) \;\middle|\; u\right)\right)- \mathbb{E}\left(\log D_{k} (M) \right)\\
			=&\mathbb{E}\left(\mathbb{E}\left(\log D_{k}\left( \left(1+2u+2u^2\right)M \right) \;\middle|\; u\right)\right)- \mathbb{E}\left(\log D_{k} (M) \right)\\
			=&k\cdot \mathbb{E}\left( \log(1+2u+2u^2)\right)\\
			\le & \frac{k u_0}{2}\int_{0}^{\infty} t^{-\frac{3}{2}}\log(1+2t+2t^2)dt\\
			\le & k u_0\int_{0}^{\infty} t^{-\frac{3}{2}}\log(1+\sqrt{2}t)dt\\
			=& 2^{\frac{5}{4}} \pi k u_0 
		\end{align*}
		for $k=m-1$ and $k=m$.
		
		By Lemma \ref{lemma-monotone}, Lemma \ref{lemma-positive} and Markov's inequality, we have 
		\begin{align*}
			&\frac{1}{4}\cdot\mathbb{P}\left(\sigma_m(M)\le \frac{u_0^2\alpha^2}{2m s}  \right)\\
			\le &\left(\frac{1}{2}-2\cdot\mathbb{P}\left( \sigma_1(M)\ge s\right)\right) \mathbb{P}\left(\sigma_m(M)\le \frac{u_0^2\alpha^2}{2m s}  \right)\\
			\le & \mathbb{E}\left(\log D_{m} \left(M+2u \left(M+M'-\mathbb{E}(M)\right)+2u^2 M' \right)\right.\\
			&\left. +\log D_{m-1} \left(M+2u \left(M+M'-\mathbb{E}(M)\right)+2u^2 M' \right)\right. \\
			&\left.- \log  D_{m}  (M) - \log  D_{m-1}  (M) \right)\\
			\le & 2^{\frac{5}{4}} \pi u_0 (2m-1) \\
			\le & 15 m u_0.
		\end{align*}
		Take $u_0=\frac{\sqrt{2}\varepsilon}{m}$, then we have $$\mathbb{P}\left(\sigma_m(M)\le \frac{\varepsilon^2 \alpha^2}{ m^3 s}\right)\le 60\sqrt{2} \varepsilon.$$
		
		Therefore, Theorem \ref{main-2} holds by taking $c_{1.6}=60\sqrt{2}$.
	\end{proof}

	\section{Proofs of Theorem \ref{main-1} and Theorem \ref{Matrix_anticoncentration}}
	\subsection{A variation of the $\varepsilon$-net theorem}
	We first prove the following variation of the $\varepsilon$-net theorem. One can compare it with \cite[Lemma 2.3]{Lecue} and \cite[Proposition 6.6]{Alon}.

	\begin{theorem}\label{theorem-epsilon}
			There exists an absolute constant $c_{3.1}>0$ such that the following statement holds.
			
			Let $(X,R)$ be a measurable range space of VC-dimension at most $d$, and let $0<\beta \le 1$. Let $x_1, \ldots, x_n$ be independent random variables in $X$ with $$n\ge \frac{c_{3.1} d}{\beta}\log \frac{2}{\beta}.$$ Then with probability at least $1 -  \exp\left(-\frac{n\beta }{c_{3.1}}\right)$, we have $$\left| \left\{ i\in \left\{1,\ldots,n \right\}: x_i \in r \right\}\right|\ge \frac{n\beta }{2}$$ for each $r\in R$ with $$\sum_{i=1}^n\mathbb{P}\left(x_i\in r\right)\ge n\beta.$$
	\end{theorem}
		\begin{proof}
			Let $y_1,\ldots, y_n$ be an independent copy of the random sequence $x_1,\ldots, x_n$. Let $x_*$ denote the collection of all $x_i$ ($1\le i\le n$), and let $y_*$ denote the collection of all $y_i$ ($1\le i\le n$).
			
			Let $R_\beta$ denote the set
			$$\left\{r\in R:\sum_{i=1}^n\mathbb{P}\left(x_i\in r\right)\ge n\beta\right\}.$$ 
			For each $r\in R_\beta$, let $E_{r,1}$ denote the event that $$\left| \left\{ i\in \left\{1,\ldots,n \right\}: x_i \in r \right\}\right|< \frac{n\beta }{2},$$
			and let $E_{r,2}$ denote the event that 
			$$\left|\left\{i\in\left\{1,\ldots, n\right\}: y_i\in r\right\}\right|\ge \frac{5n\beta}{6}.$$
			Then by Chernoff bound, we have
			$$\min_{r\in R_{\beta}} \mathbb{P}\left(E_{r,2}\right)\ge 1-\exp\left(-\frac{n\beta}{72}\right)$$	which implies
			\begin{align*}
				&\mathbb{P}\left(\bigcup_{r\in R_\beta} \left(E_{r,1}\cap E_{r,2}\right)\right)\\
				\ge&\mathbb{P}\left(\bigcup_{r\in R_\beta} E_{r,1}\right)\mathbb{P}\left(\bigcup_{r\in R_\beta} \left(E_{r,1}\cap E_{r,2}\right)\;\middle| \bigcup_{r\in R_\beta} E_{r,1}\;\right)\\
				\ge&\mathbb{P}\left(\bigcup_{r\in R_\beta} E_{r,1}\right) \inf_{r\in R_{\beta}}\mathbb{P}\left( E_{r,2}\right)\\
				\ge&\left(1-\exp\left(-\frac{n\beta}{72}\right)\right)\mathbb{P}\left(\bigcup_{r\in R_\beta} E_{r,1}\right).
			\end{align*}
		
 			Let $z_1,\ldots, z_n$ and $t_1,\ldots, t_n$ be two random sequences where each $(z_i,t_i)$ is an independent choice between $(x_i,y_i)$ and $(y_i,x_i)$ with uniform probability. Then $z_1,\ldots, z_n, t_1,\ldots, t_n$ has the same joint distribution as $x_1,\ldots, x_n,y_1,\ldots, y_n$.
 			
			For each $r\in R_\beta$, let $E_{r,3}$ denote the event that $$\left| \left\{ i\in \left\{1,\ldots,n \right\}: z_i \in r \right\}\right|< \frac{n\beta }{2},$$
			and let $E_{r,4}$ denote the event that 
			$$\left|\left\{i\in\left\{1,\ldots, n\right\}: t_i\in r\right\}\right|\ge \frac{5n\beta}{6}.$$
			
			For any $r\in R_\beta$, we denote
			\begin{align*}
				&\sum_{i=1}^n \mathbb{P}\left(z_i\in r\;\middle| \; x_*, y_*\right)\\
				=&\sum_{i=1}^n \mathbb{P}\left(t_i\in r\;\middle| \; x_*, y_*\right)\\
				=&\frac{1}{2}\left(\left| \left\{ i\in \left\{1,\ldots,n \right\}: x_i \in r \right\}\right|+\left| \left\{ i\in \left\{1,\ldots,n \right\}: y_i \in r \right\}\right|\right)
			\end{align*}
			by $\mu_r$. If $\mu_r\ge \frac{2n\beta}{3}$, then by Chernoff bound, we have
			\begin{align*}
				&\mathbb{P}\left(E_{r,3}\;\middle|\; x_*,y_*\right)\\
				\le &\exp\left(-\frac{\mu_r}{2}\left(1-\frac{n\beta}{2\mu_r}\right)^2\right)\\
				\le& \exp\left(-\frac{n\beta}{48}\right);
			\end{align*}
			if $\mu_r< \frac{2n\beta}{3}$, then by Chernoff bound, we have			
			\begin{align*}
				&\mathbb{P}\left(E_{r,4}\;\middle|\; x_*,y_*\right)\\
				\le &\exp\left(-\frac{\mu_r}{\frac{5n\beta}{6\mu_r}+1}\left(\frac{5n\beta}{6\mu_r}-1\right)^2\right)\\
				\le& \exp\left(-\frac{n\beta}{54}\right).
			\end{align*} 
			Therefore $$\mathbb{P}\left(E_{r,3}\cap E_{r,4}\;\middle|\; x_*,y_*\right)\le \exp\left(-\frac{n\beta}{54}\right)$$
			almost surely for any $r\in R_\beta$.
			
			By Sauer-Shelah-Perles lemma, for each pair $(x_*,y_*)$, there exists a subset $R_{x_*,y_*}$ of $R_\beta$ with cardinality at most $\sum_{i=0}^{d}\binom{2n}{i}$, such that for every $r\in R_\beta$ there exists $r'\in R_{x_*,y_*}$ with $$\left\{i\in{1,\ldots,n}: x_i\in r\right\}=\left\{i\in{1,\ldots,n}: x_i\in r'\right\}$$
			and $$\left\{i\in{1,\ldots,n}: y_i\in r\right\}=\left\{i\in{1,\ldots,n}: y_i\in r'\right\}.$$ 
			Then we have 
			\begin{align*}
				&\mathbb{P}\left(\bigcup_{r\in R_\beta} \left(E_{r,1}\cap E_{r,2}\right)\right)\\
				=&\mathbb{P}\left(\bigcup_{r\in R_\beta} \left(E_{r,3}\cap E_{r,4}\right)\right)\\
				=&\mathbb{E}\left(\mathbb{P}\left(\bigcup_{r\in R_\beta} \left(E_{r,3}\cap E_{r,4}\right)\;\middle|\;x_*,y_*\right)\right)\\
				=&\mathbb{E}\left(\mathbb{P}\left(\bigcup_{r\in R_{x_*,y_*}} \left(E_{r,3}\cap E_{r,4}\right)\;\middle|\;x_*,y_*\right)\right)\\
				\le &\mathbb{E}\left(\left|R_{x_*,y_*}\right|\sup_{r\in R_\beta}\mathbb{P}\left( E_{r,3}\cap E_{r,4}\;\middle|\;x_*,y_*\right)\right)\\
				\le &\exp\left(-\frac{n\beta}{54}\right) \sum_{i=0}^{d}\binom{2n}{i}\\
				\le &\exp\left(-\frac{n\beta}{54}\right)\left(\frac{4e n}{d}\right)^d.
			\end{align*}
		
		Take $c_{3.1}=1000$, then we have 
		\begin{align*}
			&\frac{\frac{4en}{d}}{\log\frac{4en}{d}}\\
			\ge&\frac{\frac{4000e}{\beta}\log \frac{2}{\beta}}{\log\left(\frac{4000e}{\beta}\log \frac{2}{\beta}\right)}\\
			\ge&\frac{240e}{\beta},
		\end{align*}
		and thus
		\begin{align*}
			&\mathbb{P}\left(\bigcup_{r\in R_\beta} E_{r,1}\right)\\
			\le&\left(1-\exp\left(-\frac{n\beta}{72}\right)\right)^{-1}\mathbb{P}\left(\bigcup_{r\in R_\beta} \left(E_{r,1}\cap E_{r,2}\right)\right)\\
			\le&\left(1-\exp\left(-\frac{n\beta}{72}\right)\right)^{-1} \exp\left(-\frac{n\beta}{54}\right)\left(\frac{4e n}{d}\right)^d\\
			\le&\left(1-\exp\left(-\frac{n\beta}{72}\right)\right)^{-1} \exp\left(-\frac{n\beta}{540}\right)\\
			\le &\exp\left(-\frac{n\beta}{c_{3.1}}\right).
		\end{align*}
	
		Therefore, Theorem \ref{theorem-epsilon} holds.
		\end{proof}
		\subsection{Proof of Theorem \ref{main-1}}
		For any unit vector $x$ in $\mathbb{R}^m$ (resp. $\mathbb{C}^m$), the set of complex matrices $$r(x):=\left\{M\in \mathbb{C}^{m\times m}: \left|x^* M x\right|> \alpha \right\}$$
		is defined by the sign of a real polynomial in $2m$ variables of degree at most $4$. By Warren's bound \cite{Warren}, the VC-dimension of the family of all $r(x)$ is at most $32m.$ Take $c_{1.1}=32c_{3.1}$, where $c_{3.1}$ is the absolute constant in Theorem \ref{theorem-epsilon}. By the small-ball probability bound
		\begin{align*}
			&\sum_{i=1}^n\mathbb{P}\left(M_i\in r(x) \right)\\
			\ge &\mathbb{P}\left(x^* M_i x>\alpha \right)\\
			\ge &n\beta
		\end{align*}
		for any unit vector $x$ in $\mathbb{R}^m$ (resp. $\mathbb{C}^m$), and Theorem \ref{theorem-epsilon}, we have
		\begin{align*}
			&\mathbb{P}\left(\sigma_m\left(\frac{1}{n}\sum_{i=1}^n M_i\right)\le \frac{\alpha\beta}{2} \right)\\
			\le&\mathbb{P}\left(\inf_{x}\left|\left\{i\in\left\{1,\ldots,n\right\}: x^* M_i x >\alpha \right\}\right|
			< \frac{n\beta}{2}\right)\\
			=&\mathbb{P}\left(\inf_{x,y}\left|\left\{i\in\left\{1,\ldots,n\right\}: M_i\in r(x) \right\}\right|< \frac{n\beta}{2}\right)\\
			\le& \exp\left(-\frac{n\beta}{c_{3.1}}\right)\\
			\le& \exp\left(-\frac{n\beta}{c_{1.1}}\right)
		\end{align*}
		where $x$ ranges over all unit vectors in $\mathbb{R}^m$ (resp. $\mathbb{C}^m$).
		
		Therefore, Theorem \ref{main-1} holds.
		\subsection{Proof of Theorem \ref{Matrix_anticoncentration}}
		For any pair $(x,y)$ in $S_{m,T}$, the set of complex matrices $$r(x,y):=\left\{M\in \mathbb{C}^{m\times m}: \left|x^* M y\right|> \alpha \right\}$$
		is defined by the sign of a real polynomial in $4m$ variables of degree at most $4$. By Warren's bound \cite{Warren}, the VC-dimension of the family of all $r(x,y)$ is at most $64m.$ Let $c_{3.1}>0$ be the absolute constant in Theorem \ref{theorem-epsilon}. For the rest of the proof, assume that $$n\ge \frac{64 c_{3.1}m}{\beta}\log \frac{2}{\beta}.$$ By Theorem \ref{theorem-epsilon} and the small-ball probability bound
		\begin{align*}
			&\sum_{i=1}^n\mathbb{P}\left(M_i\in r(x,y) \right)\\
			\ge &n\cdot P_{\alpha,T}\left(M'\right)\\
			\ge &n\beta
		\end{align*}
		for any $(x,y)$ in $S_{m,T}$, we have
		\begin{align*}
			&\mathbb{P}\left(\inf_{x,y}\left|\left\{i\in\left\{1,\ldots,n\right\}: \left|x^* M_i y\right| >\alpha \right\}\right|
			< \frac{n\beta}{2}\right)\\
			=&\mathbb{P}\left(\inf_{x,y}\left|\left\{i\in\left\{1,\ldots,n\right\}: M_i\in r(x,y) \right\}\right|< \frac{n\beta}{2}\right)\\
			\le& \exp\left(-\frac{n\beta}{c_{3.1}}\right)
		\end{align*}
		where $(x,y)$ ranges over $S_{m,T}$.
		
		Let $M_*$ denote the collection of all random matrices $M_i$ ($1\le i\le n$). Then $M$ conditioned on $M_*$ is a random matrix with jointly Gaussian entries. Let $\alpha'=\frac{1}{4}\alpha \beta^{\frac{1}{2}}$ be a positive real number. Then we have
		\begin{align*}
			&\mathbb{P}\left(P_{\alpha',T}\left(M\;\middle|\; M_*\right)<\frac{1}{2}\right)\\
			=&\mathbb{P}\left(\inf_{x,y} \mathbb{P}\left( 
			\left|x^*M y\right|>\alpha'\;\middle|\; M_* \right)<\frac{1}{2}\right)\\
			=&\mathbb{P}\left(\sup_{x,y} \mathbb{P}\left( 
			\left|x^*M y\right|\le\alpha'\;\middle|\; M_* \right)\ge \frac{1}{2}\right)\\
			\le &\mathbb{P}\left(\sup_{x,y} \mathbb{P}\left( 
			\max\left\{\left|\Re\left(x^*M y\right)\right|,\left|\Im\left(x^*M y\right)\right|\right\}\le\alpha'\;\middle|\; M_* \right)\ge \frac{1}{2}\right)\\
			\le &\mathbb{P}\left(\sup_{x,y} \frac{2\alpha'}{\sqrt{2\pi\max\left\{\mathrm{Var}\left(\Re\left(x^*M y\right)\;\middle|\;M_*\right),\mathrm{Var}\left(\Im\left(x^*M y\right)\;\middle|\;M_*\right)\right\}}}\ge \frac{1}{2}\right)\\
			\le &\mathbb{P}\left(\sup_{x,y} \frac{2\alpha'}{\sqrt{\pi\mathrm{Var}\left(x^*M y\;\middle|\;M_*\right)}}\ge \frac{1}{2}\right)\\
			=&\mathbb{P}\left(\inf_{x,y}\sum_{i=1}^n\left|x^* M_i y\right|^2\le \frac{16n\alpha'^2}{\pi}\right)\\
			\le &\mathbb{P}\left(\inf_{x,y}\sum_{i=1}^n\left|x^* M_i y\right|^2\le \frac{n\alpha^2\beta}{2}\right)\\
			\le &\mathbb{P}\left(\inf_{x,y}\left|\left\{i\in\left\{1,\ldots,n\right\}: \left|x^* M_i y\right| >\alpha \right\}\right|
			< \frac{n\beta}{2}\right)\\
			\le& \exp\left(-\frac{n\beta}{c_{3.1}}\right)
		\end{align*}
		where $(x,y)$ ranges over $S_{m,T}$.
		
		Let $c_{1.6}>0$ be the absolute constant in Theorem \ref{main-2}. Take $c_{1.8}=\max\left\{4 c_{1.6},64c_{3.1},8\right\}$ as a positive constant. Let $\varepsilon'=4\varepsilon$ be a non-negative real number. With $$n\ge \frac{c_{1.8}m}{\beta}\log \frac{2}{\beta},$$
		we have
		\begin{align*}
			&\mathbb{P}\left(\sigma_m(M)\le\frac{\varepsilon^2 \alpha^2 \beta}{m^3 s}\right)\\
			&\mathbb{P}\left(\sigma_m(M)\le\frac{\varepsilon'^2 \alpha'^2}{m^3 s}\right)\\
			=&\mathbb{E}\left(\mathbb{P}\left(\sigma_m(M)\le \frac{\varepsilon'^2 \alpha'^2}{ m^3 s} \right)\;\middle|\; M_*\right)\\
			\le& c_{1.6}\varepsilon'+\mathbb{P}\left(P_{\alpha',T}\left(M\;\middle|\; M_*\right)<\frac{1}{2}\right)+\mathbb{P}\left(\mathbb{P}\left(\sigma_1(M)>s\;\middle| \;M_*\right)>\frac{1}{8}\right)\\
			\le& 4c_{1.6}\varepsilon+ \exp\left(-\frac{n\beta}{c_{3.1}}\right)+8\cdot \mathbb{E}\left(\mathbb{P}\left(\sigma_1(M)>s\;\middle| \;M_*\right)\right)\\
			=& 4c_{1.6}\varepsilon+ \exp\left(-\frac{n\beta}{c_{3.1}}\right)+8\cdot \mathbb{P}\left(\sigma_1(M)>s\right)\\
			\le&c_{1.8} \left(\varepsilon+\mathbb{P}\left(\sigma_1(M)>s\right)\right)+ \exp\left(-\frac{n\beta}{c_{1.8}}\right)
		\end{align*}
		for each $\varepsilon\ge 0$ and $s\ge 0$.
		
		Therefore, Theorem \ref{Matrix_anticoncentration} holds.

	\section{Proof of Theorem \ref{main-4}}
	For convenience sake, we assume
	$$n> 10000 m^2\log n+1$$ and $$h\ge 10000 m \log n \log\frac{1}{\alpha}$$ throughout this section.
	
	We decompose the unit sphere $S^{n-1}$ into $\ceil{\frac{\log m}{\log 10}}+2$ classes of vectors according to their data compression ratio. In order to state our decomposition, we define the classes of $A$-compressible and $A$-incompressible vectors as follows.
	
	\begin{definition}\label{def-com}
		For each $k=1,\ldots,m$, $l=0,1,\ldots, m$ and $\varepsilon>0$, we define the set $\mathrm{Comp}_A(k, l, \varepsilon)$ of $A$-compressible vectors as the set of all $w$ in $S^{n-1}$ such that there exist $x\in S^{k-1}$ and $y\in \mathbb{R}^{n}$ with $\left|\mbox{Supp}(y)\right|\le l$ and 
		$$\norm{\sum_{i=1}^k x_i A_i^T w-y}_2<\varepsilon$$
	\end{definition}
	
	\begin{definition} 
		The set $\IncompA{k,l,\varepsilon}$ of $A$-incompressible vectors is defined as $S^{n-1}\setminus \CompA{k,l,\varepsilon}$. 
	\end{definition}

	For each $j=0,1,\ldots, \ceil{\frac{\log m}{\log 10}}$, we consider the decomposition 
	$$S^{n-1}=\IncompA{k_j,l_j,\varepsilon_j}\cup\CompA{k_j,l_j,\varepsilon_j}$$
	of $S^{n-1}$ into $A$-compressible and $A$-incompressible vectors, where the tuple $\left(k_j,l_j,\varepsilon_j\right)$ is defined as $$\left(k_j,l_j,\varepsilon_j\right):=\left(\left\lceil\frac{m}{10^{j}}\right\rceil,\left\lfloor\frac{n}{10^{j+2} \log n}\right\rfloor,\alpha^{10^{\ceil{\frac{\log m}{\log 10}}-j+1}}\right).$$

    The following lemma treats the $\CompA{k_{\ceil{\frac{\log m}{\log 10}}},l_{\ceil{\frac{\log m}{\log 10}}},\varepsilon_{\ceil{\frac{\log m}{\log 10}}}}$ part. It is essentially a quantitative version of \cite[Lemma 5.6]{PV}. 
    
    \begin{lemma}\label{lemma-base}
    	Let $G_0$ denote an $n\times \left\lfloor\frac{n-1}{2m}\right\rfloor$ sparse Gaussian matrix where entry is set to $\mathcal{N}(0,1)$ with probability $\frac{h}{n}$, and $0$ otherwise. Then with probability at least $1-\alpha^{10}$, we have $$\norm{G_0^T A_1^T w}_2\ge \alpha^{7}$$ for every vector $w\in \CompA{1,\floor{\frac{n}{100m \log n}},\alpha^{10}}$.
    \end{lemma}
	\begin{proof}
		Let $s$ denote the real number $\exp\left(\frac{\log\frac{1}{\alpha}}{\log n}\right)$, then by the condition $\alpha\le n^{-\log n}$, we have $s\ge n$. For a real standard normal random variable $x$, we have $$\mathbb{P}\left(|x|\ge s\right)\le\frac{2}{s\sqrt{2\pi}}\exp\left(-\frac{s^2}{2}\right).$$
		Since
		$$s\ge \frac{10\log \frac{1}{\alpha}}{\log n}\ge 10\sqrt{\log \frac{1}{\alpha}},$$
		by a union bound, we have 
		\begin{align*}
			&\mathbb{P}\left(\norm{G_0}_\infty \ge s\right)\\
			\le&\frac{2n^2}{s\sqrt{2\pi}}\exp\left(-\frac{s^2}{2}\right)\\
			\le&\frac{2n^2}{s\sqrt{2\pi}}\alpha^{50}\\
			\le&\alpha^{49}.
		\end{align*}
		
		For every non-negative integer $i$ with $2^i\le \frac{n}{50 m\log n}$, let $C_i\subseteq S^{n-1}$ denote the set of unit vectors with less than $2^i$ entries with magnitude at least $s^{-4(i+1)}$, and let $S_i\subseteq S^{n-1}$ denote the set of unit vectors with less than $2^i$ nonzero entries. Then we have $C_0=S_0=\emptyset$.
		
		For each positive integer $i$ with $2^{i}\le \frac{n}{50 m\log n}$, let $N_i$ be an $s^{-4(i+1)}$-net of $S_i\setminus C_{i-1}$ with size at most
		\begin{align*}
			&\binom{n}{2^i-1}\left(1+\frac{2}{s^{-4(i+1)}}\right)^{2^i-1}\\
		\le &s^{2^i(4i+6)}\\
		\le &s^{6\cdot 2^i \log n}\\
		=&\alpha^{-6\cdot 2^i}.
		\end{align*}
		Then $N_i$ is a $\left(2s^{-4i-\frac{7}{2}}\right)$-net of $C_i\setminus C_{i-1}$.
		
		Let $g$ denote an $n$-dimensional sparse Gaussian vector where each entry is set to $\mathcal{N}(0,1)$ with probability $\frac{h}{n}$. Then for every $y\in N_i$, we have
		\begin{align*}
			&\mathbb{P}\left(\norm{G_0^T y }_\infty <  s^{-4 i-1}\right)\\
		=&\mathbb{P}\left(\left|g^T y\right| <  s^{-4 i-1}\right)^{\floor{\frac{n-1}{2m}}}\\
		\le &\mathbb{P}\left(\left|g^T y\right| <  s^{-4 i-1}\right)^{\frac{n}{3m}}\\
		\le&\left(\mathbb{P}\left(\left|g^T y\right| < s^{-4 i-1}\;\middle|\;\mathrm{Var}\left(g^T y\right)\ge s^{-8 i} \right)+\mathbb{P}\left(\mathrm{Var}\left(g^T y\right)< s^{-8 i} \right)\right)^{\frac{n}{3m}}\\
		\le&\left(\frac{1}{s}+\left(1-\frac{h}{n}\right)^{2^{i-1}}\right)^{\frac{n}{3m}}\\
		\le&\max\left\{\frac{1}{\sqrt{s}},\left(1-\frac{h}{n}\right)^{2^{i-2}} \right\}^{\frac{n}{3m}}\\
		\le& \max\left\{s^{-\frac{n}{6m}},\exp\left(-\frac{2^i h}{12m}\right)  \right\}.
		\end{align*}
		By a union bound, we have
		\begin{align*}
			&\mathbb{P}\left(\inf_{y\in  N_i } \norm{G_0^T y }_\infty <  s^{-4 i-1}\right)\\
			\le & \alpha^{-6\cdot 2^i} \max\left\{s^{-\frac{n}{6m}},\exp\left(-\frac{2^i h}{12m}\right)  \right\}\\
			\le &\max\left\{s^{-\frac{n}{25m}}, \exp\left(-\frac{2^i h}{25m}\right)\right\}\\
			\le &\exp\left(-\frac{h}{25m}\right)\\
			\le &\alpha^{40}.
		\end{align*}

		By another union bound, with probability at least $1-\alpha^{10}$, we have $\norm{G_0}_\infty<s$ and $$\inf_{y\in  N_i } \norm{G_0^T y }_\infty \ge s^{-4 i-1}$$
		for each positive integer $i$ with $2^{i}\le \frac{n}{50 m\log n}$. Under these conditions, we will prove that
		$$\norm{G_0^T A_1^T w}_2\ge \alpha^{10}$$ for every vector $w\in \CompA{1,\floor{\frac{n}{100 m \log n}},\alpha^{10}}$.
		
		By the definition of $A$-compressible vectors, the vector $A_1^Tw$ has at most $\frac{n}{100 m \log n}$ entries with magnitude at least $\alpha^{10}$ for every $w\in \CompA{1,\floor{\frac{n}{100 m \log n}},\alpha^{10}}$. Because $$\norm{A_1^T w}_2\ge\sigma_n\left(A_1\right) \ge \alpha,$$
		the vector $\frac{A_1^Tw}{\norm{A_1^T w}_2}$ has at most $\frac{n}{100 m \log n}$ entries with magnitude at least $\alpha^{9}=s^{-9\log n}$. This there exists a positive integer $i$ with $2^i\le \frac{n}{50m\log n}$ such that $\frac{A_1^Tw}{\norm{A_1^T w}_2}\in C_i\setminus C_{i-1}.$
		Because $N_i$ is a $\left(2s^{-4i-\frac{7}{2}}\right)$-net of $C_i\setminus C_{i-1}$, there exists $y\in N_i$ with $$\norm{\frac{A_1^Tw}{\norm{A_1^T w}_2}-y}_2\le2s^{-4i-\frac{7}{2}}.$$
		Then we have
		\begin{align*}
			 &\norm{G_0^T A_1^T w}_2\\
			= &\norm{A_1^T w}\norm{G_0^T \frac{A_1^T w}{\norm{A_1^T w}_2}}_2\\
			\ge &\alpha\norm{G_0^T \frac{A_1^T w}{\norm{A_1^T w}_2}}_2\\
			\ge &\alpha\left(\norm{G_0^T y}_2 -\norm{G_0^T\left( \frac{A_1^T w}{\norm{A_1^T w}_2}-y\right)}_2\right)\\
			\ge &\alpha \left(\norm{G_0^T y}_\infty -\sigma_1\left(G_0\right)\norm{ \frac{A_1^T w}{\norm{A_1^T w}_2}-y}_2\right)\\
			\ge &\alpha\left(s^{-4i-1}-2ns^{-4i-\frac{7}{2}}\norm{G_0}_\infty\right)\\
			\ge&\alpha s^{-4i-\frac{3}{2}}\left(\sqrt{s}-2\right)\\
			\ge&  \alpha^{7}.
		\end{align*}
		
		Therefore, Lemma \ref{lemma-base} holds.
	\end{proof}
	
	Our next technical lemma handles the incompressible part and each inductive step in the compressible part simultaneously. 
	\begin{lemma}\label{lemma-d}
		There exists an absolute constant $c_{4.4}>0$ such that the following statement holds.
		
		Let $g_1,\ldots, g_{d}$ be independent copies of the $n$-dimensional sparse Gaussian vector $g$ where each entry is set to $\mathcal{N}(0,1)$ with probability $\frac{h}{n}$, and $0$ otherwise. Suppose that $$h\ge c_{4.4} m \log n \log\frac{1}{\alpha}.$$
		Suppose that $k\in\left\{1,\ldots, m\right\}$, $l\in \left\{0,1,\ldots, n\right\}$ and $0<\varepsilon\le \alpha^{10k}$ satisfy
		$$\frac{l+1}{k}\ge \frac{n}{200m\log n}.$$
		
		Let $B\subseteq \mathbb{R}^n$ be a subspace of dimension $d'$. For each $i\in \left\{1,\ldots ,d\right\}$ and each $k'\in\{1,\ldots, k\}$, we define an $n\times k$ random matrix $K'_{k',i}$ as
		$$K'_{k',i}:=\left[
		\begin{array}{c|c|c|c}
			A_1 g_i & A_2 g_i & \ldots & A_{k'} g_i
		\end{array}
		\right].$$ 
		
		Let $E_{d,1}$ denote the event that either of the following holds:
		\begin{enumerate}[label=(\arabic*)] 
			\item we have $S^{n-1}\cap B\subseteq \IncompA{k,l,\varepsilon} ;$
			\item we have $B=\mathbb{R}^n$, $k=m$, and
			$$\left\{w\in S^{n-1} : K_{\frac{n}{m},i}'^T w=0\mbox{ for all }i\in I \right\} \subseteq \IncompA{k,l,\varepsilon}$$for every subset $I\subseteq \left\{1,\ldots, d\right\}$ with $|I|=\frac{n}{m}-1$.
		\end{enumerate}
		Let $E_{d,2}$ denote the event that there exists $w\in S^{n-1}\cap B$ such that $$\sum_{i=1}^d\norm{K_{k,i}'^T w}^2_2< \varepsilon^{10}.$$
		Then we have  
		$$\mathbb{P}\left(E_{d,1}\cap E_{d,2}\right)\le \alpha^{7kd-7k\ceil{\frac{d'}{k}}+7k}.$$
	\end{lemma}
	\begin{proof}
		We induct on $d$. The base case where $d<\frac{d'}{k}$ is trivial. Now we assume that $d\ge \frac{d'}{k}$ and Lemma \ref{lemma-d} holds for $d-1$.
		
		Let $U_B$ be an $n\times d'$ matrix whose columns are an orthonormal basis for $B$. Let $i$ be any positive integer with $i\le \ceil{\frac{d'}{k}}$. Define $d_i$ as
		$$d_i:=\begin{cases}
			k&\mbox{, if } 1\le i< \ceil{\frac{d'}{k}} \\
			d'-k\ceil{\frac{d'}{k}}+k&\mbox{, otherwise.} 
		\end{cases}$$ Define $N_i$ to be the $d'\times d_i$ matrix $U_B^T K'_{d_i,i}$. Let $N$ be the $d'\times d'$ matrix whose columns are those of $N_1,\ldots, N_{\ceil{\frac{d'}{k}}}$. Let $U_i$ denote a $d'\times d_i$ matrix whose columns forms an orthonormal set orthogonal to all column vectors of $N$ except for those of $N_i$. We may assume that $U_i$ is independent with $g_i$. Let $E_i$ denote the event that $$\sigma_{d_i}\left(U_i^T N_i\right)<\sqrt{n} \varepsilon^5.$$ 
		
		Suppose that $E_{d,2}$ happens. Then there exists $w\in S^{d'-1}$ with $\norm{N^T w}< \varepsilon^5$. Because $\sigma_{d'}(N)=\sigma_{d'}\left(N^T\right)$, there exists $x\in S^{d'-1}$ with $\norm{N x}<\varepsilon^5$. Equivalently, there exists a vector $x_i\in \mathbb{R}^{d_i}$ for each $i$ with $1\le i\le \ceil{\frac{d'}{k}}$ such that $\sum_{i=1}^{\ceil{\frac{d'}{k}}} \norm{x_i}_2^2=1$ and
		$\norm{\sum_{i=1}^{\ceil{\frac{d'}{k}}} N_i x_i}_2< \varepsilon^5,$
		which implies 
		$\norm{U_i^T N_ix_i}_2< \varepsilon^5$
		for each $i$ with $1\le i\le \ceil{\frac{d'}{k}}$.
		Therefore, the event $E_{d,2}$ implies $\bigcup_{i=1}^{\ceil{\frac{d'}{k}}} E_i$.
		
		Suppose that $E_{d,1}\cap E_{d,2}$ happens. Then there exists a positive integer $i$ with $i\le \ceil{\frac{d'}{k}}$ such that $E_{d,1}\cap E_{d,2} \cap E_i$ happens. By deleting all columns depending on $g_i$, the corresponding event $E_{d-1,1}\cap E_{d-1,2}$ (note that this event implicitly depends on $i$ and is independent with $g_i$) happens. Let $g_{*\setminus i}$ denote the collection of all random vectors $g_1,\ldots, g_d$ except for $g_i$. Conditioned on $g_{*\setminus i}$, the $d_i\times d_i$ random matrix $U_i^T N_i$ is a linear combination of $n$ deterministic matrices with coefficients following sparse Gaussian distribution. In other words, we may write $U_i^T N_i$ as $$U_i^T N_i= \sum_{j=1}^n g_{i,j} M_j,$$
		where $g_{i,1},\ldots, g_{i,n}$ are the coordinates of the sparse Gaussian vector $g_i$ and each $M_j$ $(1\le j\le n)$ is independent with $g_i$. The event $E_{d-1,1}$ implies the global small-ball probability bound $$\inf_{x,y}\left|\left\{j\in\left\{1,\ldots, n\right\}: \left|x^* M_j y\right|\ge \frac{\varepsilon}{\sqrt{n}} \right\}\right|>l$$
		where $x$ and $y$ range over $S^{d_i-1}$. By Theorem \ref{Matrix_anticoncentration}, there exists an absolute constant $c_{1.8}>0$ such that the event $E_{d-1,1}$ and
		$$(l+1) h\ge c_{1.8} nd_i\log \frac{n^2}{(l+1) h}$$
		imply
		\begin{align*}
			&\mathbb{P}\left(\sigma_{d_i}\left(U_i^T N_i\right)\le \frac{\varepsilon'^2\varepsilon^2 (l+1) h}{n^{2}  d_i^3 s}\;\middle|\; g_{*\setminus i}\right)\\
			\le&c_{1.8}\left(\varepsilon'+\mathbb{P}\left(\sigma_1\left(U_i^T N_i\right)>s\;\middle|\; g_{*\setminus i}\right)\right)+ \exp\left(-\frac{(l+1)h}{c_{1.8}n}\right)
		\end{align*}
		for each $\varepsilon'\ge 0$ and $s\ge 0$. Take $c_{4.4}=\max\left\{1600 c_{1.8}, 10000\right\}$, $\varepsilon'=\alpha^{8k}$ and $s=\alpha^{-8k}$. Then by
		\begin{align*}
			&(l+1) h\\
			\ge &\frac{c_{4.4} nk\log \frac{1}{\alpha} }{200}\\
			\ge &8 c_{1.8} n k \log \frac{1}{\alpha}\\
			\ge & c_{1.8} n d_i\log \frac{n^2}{(l+1)h},
		\end{align*} 
		the event $E_{d-1,1}$ implies
		\begin{align*}
			&\mathbb{P}\left(E_i\;\middle|\; g_{*\setminus i}\right)\\
			=&\mathbb{P}\left(\sigma_{d_i}\left(U_i^T N_i\right)< \sqrt{n}\varepsilon^5\;\middle|\; g_{*\setminus i} \right)\\
			\le&\mathbb{P}\left(\sigma_{d_i}\left(U_i^T N_i\right)\le \frac{ \alpha^{24k}\varepsilon^2 }{n^5 }\;\middle|\; g_{*\setminus i}\right)\\
			\le&\mathbb{P}\left(\sigma_{d_i}\left(U_i^T N_i\right)\le \frac{\varepsilon'^2\varepsilon^2 (l+1) h}{n^{2}  d_i^3 s}\;\middle|\; g_{*\setminus i}\right)\\
			\le&c_{1.8}\left(\varepsilon'+\mathbb{P}\left(\sigma_1\left(U_i^T N_i\right)>s\;\middle|\; g_{*\setminus i}\right)\right)+ \exp\left(-\frac{(l+1)h}{c_{1.8}n}\right)\\
			\le& n\left(\alpha^{8k} + \mathbb{P}\left(\sigma_1\left( K'_{d_i,i}\right)>\frac{1}{\alpha^{8k}}\right) \right)+\alpha^{8k}\\
			\le& n\left(\alpha^{8k} + \mathbb{P}\left( \norm{g_i}_2>\frac{1}{\sqrt{n}\alpha^{8k}}\right) \right)+\alpha^{8k}\\
			\le& n\left(\alpha^{8k} + \mathbb{P}\left(\norm{g_i}_\infty>\frac{1}{n\alpha^{2k}}\right) \right)+\alpha^{8k}\\
			\le& n\left(\alpha^{8k} + \frac{2 n^2 \alpha^{8k}}{\sqrt{2\pi}} \exp\left(-\frac{1}{2n^2 \alpha^{16k}}\right)\right)+\alpha^{8k}\\
			\le&n^3 \alpha^{8k}.
		\end{align*}
		
		Therefore, by inductive hypothesis, we have 
		\begin{align*}
			 &\mathbb{P}\left(E_{d,1}\cap E_{d,2}\right)\\
			\le&\sum_{i=1}^{\ceil{\frac{d'}{k}}} \mathbb{P}\left(E_{d,1}\cap E_{d,2}\cap E_i\right)\\
			\le&\sum_{i=1}^{\ceil{\frac{d'}{k}}} \mathbb{P}\left(E_{d-1,1}\cap E_{d-1,2}\cap E_i\right)\\
			= &\sum_{i=1}^{\ceil{\frac{d'}{k}}}\mathbb{E}\left(\mathbb{P}\left(E_{d-1,1}\cap E_{d-1,2}\cap E_i\;\middle|\; g_{*\setminus i} \right)\right)\\
			= &\sum_{i=1}^{\ceil{\frac{d'}{k}}}\mathbb{E}\left(\mathbbm{1}_{E_{d,1}\cap E_{d,2}} \cdot \mathbb{P}\left(E_i\;\middle|\; g_{*\setminus i} \right)\right)\\
			\le&n^3 \alpha^{8k}\sum_{i=1}^{\ceil{\frac{d'}{k}}}\mathbb{E}\left(\mathbb{P}\left(\mathbbm{1}_{E_{d,1}\cap E_{d,2}}\right)\right)\\
			=&n^3 \alpha^{8k}\sum_{i=1}^{\ceil{\frac{d'}{k}}}\mathbb{P}\left(E_{d,1}\cap E_{d,2}\right)\\
			\le&n^4 \alpha^{7kd-7k\ceil{\frac{d'}{k}}+8k}\\
			\le&\alpha^{7kd-7k\ceil{\frac{d'}{k}}+7k}.
		\end{align*}
	
	By mathematical induction, Lemma \ref{lemma-d} holds.
	\end{proof}
	
	Instead taking a union bound over all $A$-compressible vectors, we have to group the $A$-compressible vectors into low dimensional linear subspaces to utilize the matrix anti-concentration inequality.
	
	\begin{lemma}\label{lemma-eps-net}
		Suppose that $k\in\left\{1,\ldots, m\right\}$, $l\in \left\{0,1,\ldots, n\right\}$ and $0<\varepsilon\le \alpha^{3k}$. There exists a union $B \subseteq \mathbb{R}^n$ of at most $\binom{n}{l}\varepsilon^{-2k}$ linear subspaces of dimension at most $k+l-1$, such that for every $w \in \CompA{k,l,\varepsilon}$ there exists $w'\in S^{n-1}\cap B$ with $$\norm{w-w'}_2\le \alpha^{-3k}\varepsilon.$$
	\end{lemma}
	\begin{proof}
		For each $l$-subset $L$ of the standard basis $\{e_1,\ldots,e_n\}$ of $\mathbb{R}^n$, let $U_L$ be the $n\times (n-l)$ matrix where $\{e_1,\ldots,e_n\}\setminus L$ is the set of columns. For each $x=(x_1,\ldots, x_k)\in S^{k-1}$ and $l$-subset $L$ of $\{e_1,\ldots,e_n\}$, let $U_{x,L}$ be an $(n-l)\times (n-k-l+1)$ matrix with orthonormal columns, such that
		$$\sigma_{n-k-l+1}\left(\sum_{i=1}^k x_i A_i U_L U_{x,L}\right)\\
		=\sigma_{n-k-l+1}\left(\sum_{i=1}^k x_i A_i U_L\right).$$
		
		Let $A_{x,L}$ denote the $n\times (n-k-l+1)$ matrix $$\sum_{i=1}^k x_i A_i U_L U_{x,L}.$$
		Then we have
		\begin{align*}
			&\sigma_{n-k-l+1}\left( A_{x,L}\right)\\
			=&\sigma_{n-k-l+1}\left(\sum_{i=1}^k x_i A_i U_L U_{x,L}\right)\\
			=&\sigma_{n-k-l+1}\left(\sum_{i=1}^k x_i A_i U_L\right)\\
			\ge&\sigma_{n-k+1}\left(\sum_{i=1}^k x_i A_i \right)\sigma_{n-l}\left(U_L\right)\\
			\ge& \alpha^{k}.
		\end{align*}
		In particular, we have $\rank\left(A_{x,L}\right)=n-k-l+1$.
		
		Let $N$ be an $\varepsilon$-net over $S^{k-1}$ with size at most $$\left(1+\frac{2}{\varepsilon}\right)^{k}\le  \varepsilon^{-2k}.$$ 
		Then the subset $B\subseteq \mathbb{R}^n$ defined as
		$$B:=\bigcup_{x\in N,L\subseteq \{e_1,\ldots,e_n\}, |L|=l} \left\{w\subseteq\mathbb{R}^n :A_{x,L}^T w=0\right\}$$
		is a union of at most $\binom{n}{l}\varepsilon^{-2k}$ linear subspaces of dimension at most $k+l-1$. 
		
		By the definition of $A$-compressible vectors, for any $w\in\CompA{k,l,\varepsilon}$, there exist $x\in S^{k-1}$ and $y\in \mathbb{R}^n$ with $\left|\mbox{Supp}(y)\right|\le l$ and 
		$$\norm{\sum_{i=1}^k x_i A_i^T w-y}_2<\varepsilon.$$
		Since $\left|\mbox{Supp}(y)\right|\le l$, there exists an $l$-subset $L$ of $\left\{e_1,\ldots, e_n\right\}$ with $U_L^Ty=0$. Since $N$ is an $\varepsilon$-net, there exists $x'\in N$ with $\norm{x-x'}_2\le \varepsilon$. Define the vector $w'' \in\mathbb{R}^n$ by 
		$$w'':=w- A_{x',L}\left(A_{x',L}^TA_{x',L}\right)^{-1} A_{x',L}^Tw.$$
		Then we have
		\begin{align*}
			&\norm{w-w''}_2\\
			=&\norm{A_{x',L}\left(A_{x',L}^TA_{x',L}\right)^{-1} A_{x',L}^Tw}_2\\
			\le&\sigma_1\left(A_{x',L}\left(A_{x',L}^TA_{x',L}\right)^{-1} \right)\norm{A_{x',L}^Tw}_2\\
			\le&\sigma_1\left(A_{x',L}\right)\sigma_{n-k-l+1}\left(A_{x',L}\right)^{-2}\norm{A_{x',L}^Tw}_2\\
			\le&\alpha^{-2k}\sigma_1\left(\sum_{i=1}^k x'_i A_i U_L  U_{x',L}\right)\norm{U_{x',L}^T U_L^T\sum_{i=1}^k x'_i A_i^T w}_2\\
			\le&\alpha^{-2k}\sigma_1\left(\sum_{i=1}^k x'_i A_i \right)\sigma_1\left(U_L\right)\sigma_1\left(U_{x',L}\right)^2\norm{ U_L^T\sum_{i=1}^k x'_i A_i^T w}_2\\
			\le &\alpha^{-2k}\left(\sum_{i=1}^k x'_i\sigma_1\left(A_i\right)\right)\norm{ U_L^T\sum_{i=1}^k x'_i A_i^T w}_2\\
			\le &\sqrt{k}\alpha^{-2k}\norm{ U_L^T\sum_{i=1}^k x'_i A_i^T w}_2\\
			\le &\sqrt{k}\alpha^{-2k}\left(\norm{ U_L^T\sum_{i=1}^k x_i A_i^T w}_2+\norm{ U_L^T\sum_{i=1}^k \left(x_i-x'_i\right) A_i^T w}_2\right)\\
			\le &\sqrt{k}\alpha^{-2k}\left(\norm{ U_L^T\left(\sum_{i=1}^k x_i A_i^T w-y\right)}_2+\sigma_1\left( U_L^T\sum_{i=1}^k \left(x_i-x'_i\right) A_i^T \right)\right)\\
			\le & \sqrt{k}\alpha^{-2k}\left(\sigma_1\left(U_L\right)\norm{\sum_{i=1}^k x_i A_i^T w-y}_2+ \sigma_1\left(U_L\right)\sum_{i=1}^k \left|x_i-x'_i\right|\sigma_1\left(A_i\right)\right)\\
			\le&\sqrt{k}\alpha^{-2k}\left(\varepsilon+\sqrt{k}\varepsilon\right)\\
			\le &2k\alpha^{-2k}\varepsilon.
		\end{align*}
		Define the vector $w'\in\mathbb{R}^n$ as $\frac{1}{\norm{w''}} w''$. Then $w'\in B$, and we have		
		\begin{align*}
			&\norm{w-w'}_2\\
			\le& \norm{w-w''}+\norm{w''-w'}\\
			=& \norm{w-w''} +\left|\norm{w''}_2-1\right|\\
			\le&2\norm{w-w''}_2\\
			\le&4k\alpha^{-2k}\varepsilon\\
			\le&\alpha^{-3k}\varepsilon.
		\end{align*}
		
		Therefore, Lemma \ref{lemma-eps-net} holds.
		
	\end{proof}

	Now we handle the compressible part by induction.
	
	\begin{lemma}\label{lemma-mid}
	There exists an absolute constant $c_{4.6}>0$ such that the following statement holds.
	
	Let $j_1$ and $j_2$ be two non-negative integers with $j_1+j_2=\ceil{\frac{\log m}{\log 10}}$. Let $G_{j_1}$ denote an $n\times \left\lfloor\frac{n-1}{2m}+\frac{j_1 (n-1)\log 10}{2m\log (10m)}\right\rfloor$ sparse Gaussian matrix where each entry is set to $\mathcal{N}(0,1)$ with probability $\frac{h}{n}$, and $0$ otherwise. 
	
	Suppose that $$h\ge c_{4.6} m \log n \log\frac{1}{\alpha}.$$
	Then with probability at least $1-\left(j_1+1\right)\alpha^{10}$, the matrix
	$$K_{j_1}:=\left[
	\begin{array}{c|c|c|c}
		A_1 G_{j_1} & A_2 G_{j_1} & \ldots & A_m G_{j_1}
	\end{array}
	\right]$$
	satisfies 
	$$\norm{K_{j_1}^T w}_2 \ge \varepsilon_{ j_2}^{\frac{7}{10}}$$
	for every vector $w\in \CompA{k_{j_2},l_{j_2},\varepsilon_{j_2}}$.
	\end{lemma}
	\begin{proof}
		We induct on $j_1$. 
		
		In the base case where $j_1=0$, we have
		\begin{align*}
			&\CompA{k_{j_2},l_{j_2},\varepsilon_{j_2}}\\
			=&\CompA{1,l_{j_2},\alpha^{10}}\\
			\subseteq& \CompA{1,\floor{\frac{n}{100m \log n}},\alpha^{10}}.
		\end{align*} 
		Hence by Lemma \ref{lemma-base}, with probability at least $1-\alpha^{10}$, we have
	$$\norm{K_0^T w}_2
	\ge\norm{G_0^T A_1^T w}_2
	\ge\alpha^7
	=\varepsilon_{j_2}^{\frac{7}{10}}$$
		for every $w\in \CompA{k_{j_2},l_{j_2},\varepsilon_{j_2}}$. Therefore Lemma \ref{lemma-mid} holds for $j_1=0$.
		
		Now assume that Lemma \ref{lemma-mid} holds for an integer $j_1$ with $0\le j_1< \ceil{\frac{\log m}{\log 10}}$. We consider the $j_1+1$ case. 
		
		Define the integer $d$ as	
		\begin{align*}
			d:=&\left\lfloor\frac{n-1}{2m}+\frac{j_1 (n-1)\log 10}{2m\log (10m)}\right\rfloor-\left\lfloor\frac{n-1}{2m}+\frac{(j_1-1) (n-1)\log 10}{2m\log (10m)}\right\rfloor\\
			\ge& \frac{(n-1)\log 10}{2m\log (10m)}-1\\
			\ge& \frac{n}{m \log n}.
		\end{align*}
		Let $g_1,\ldots, g_{d}$ be independent copies of the $n$-dimensional sparse Gaussian vector $g$ where each entry is set to $\mathcal{N}(0,1)$ with probability $\frac{h}{n}$, and $0$ otherwise. We regroup the columns of $K_{j_1+1}$ into $K'_1,\ldots,K'_d$ and $K_{j_1}$ where $$K_i':=\left[
		\begin{array}{c|c|c|c}
			A_1 g_i & A_2 g_i & \ldots & A_m g_i
		\end{array}
		\right].$$
		
		Since $\varepsilon_{j_2-1}\le \alpha^{10k_{j_2-1}}$, by Lemma \ref{lemma-eps-net}, there exists a union $B_0 \subseteq \mathbb{R}^n$ of at most $\binom{n}{l_{j_2-1}}\varepsilon_{j_2-1}^{-2k_{j_2-1}}$ linear subspaces of dimension at most $k_{j_2-1}+l_{j_2-1}-1$, such that for every $w \in \CompA{k_{j_2-1},l_{j_2-1},\varepsilon_{j_2-1}}$ there exists $w'\in S^{n-1}\cap B_0$ with $$\norm{w-w'}_2\le \alpha^{-3k_{j_2-1}}\varepsilon_{j_2-1}\le \varepsilon_{j_2-1}^{\frac{7}{10}}=\varepsilon_{j_2}^7.$$
		For every linear subspace $L$ appeared in the definition of $B_0$, let $U$ be an $n\times \dim L$ matrix whose columns form an orthogonal basis of $L$, let $L_1$ (resp. $L_2$) be the subspace of $L$ spanned by all left-singular vectors of $U K_{j_1}$ for which the corresponding singular values are at least $\varepsilon_{j_2}^{\frac{7}{10}}$ (resp. less than $\varepsilon_{j_2}^{\frac{7}{10}}$). Then $L=L_1\oplus L_2$ is an orthogonal decomposition of $L$. For each $w_1\in S^{n-1}\cap L_1$ and each $w_2\in S^{n-1}\cap L_2$, we have $$\norm{K_{j_1}^T w_1}_2\ge \varepsilon_{j_2}^{\frac{7}{10}},$$
		$$\norm{K_{j_1}^T w_2}_2< \varepsilon_{j_2}^{\frac{7}{10}},$$
		and
		$$w_1^T K_{j_1}K_{j_1}^T w_2=0.$$
		Let $B_1\subseteq \mathbb{R}^n$ (resp. $B_2\subseteq \mathbb{R}^n$) denote the union of every linear subspace $L_1$ (resp. $L_2$) where $L$ appears in the definition of $B_0$. Note that $B_0$ is deterministic while $B_1$ and $B_2$ are random.
		
		Let $E$ denote the event that $$S^{n-1} \cap B_2\subseteq  \IncompA{k_{j_2}, l_{j_2}, \varepsilon_{j_2}}.$$
		Then $E$ is independent with $K'_1,\ldots, K'_d$. By inductive hypothesis, we have $$\mathbb{P}\left(E\right)\ge 1-\left(j_1+1\right)\alpha^{10}.$$
		Take $c_{4.6}=\max\left\{c_{4.4}, 10000\right\}$, where $c_{4.4}$ is the absolute constant in Lemma \ref{lemma-d}. Because $B_2$ is a union of at most $\binom{n}{l_{j_2-1}}\varepsilon_{j_2-1}^{-2k_{j_2-1}}$ linear subspaces of dimension at most $k_{j_2-1}+l_{j_2-1}-1$, by Lemma \ref{lemma-d}, the event $E$ implies 
		\begin{align*}
			&\mathbb{P}\left(\inf_{w\in S^{n-1} \cap B_2 }\sum_{i=1}^d \norm{K'^T_i w}_2^2<\varepsilon_{ j_2}^{10}\;\middle|\; B_2\right)\\
			\le&\binom{n}{l_{j_2-1}}\varepsilon_{j_2-1}^{-2k_{j_2-1}} \alpha^{7k_{j_2}d-7k_{j_2}\ceil{\frac{k_{j_2-1}+l_{j_2-1}-1}{k_{j_2}}}+ 7k_{j_2}}\\
			\le& n^{l_{j_2-1}}\varepsilon_{j_2}^{-200 k_{j_2}}\alpha^{
		\frac{7k_{j_2} n}{m\log n}-7\left(k_{j_2-1}+l_{j_2-1}\right)	
		}\\
	\le& \varepsilon_{j_2}^{-207 k_{j_2}}\alpha^{
	\frac{7k_{j_2} n}{m\log n}-8l_{j_2-1}	
}\\
		\le& \varepsilon_{j_2}^{-207 k_{j_2}}\alpha^{
		\frac{7k_{j_2} n}{m\log n}-\frac{8n}{10^{j_2+1}\log n}
	 	}		\\
	\le&\varepsilon_{j_2}^{-207 k_{j_2}}\alpha^{
\frac{6k_{j_2} n}{m\log n}}	\\
	\le&\alpha^{
\frac{6k_{j_2} n}{m\log n}-20700 k_{j_2}^2}\\
\le&\alpha^{
	\frac{3k_{j_2} n}{m\log n}}\\
\le&\alpha^{20}.
		\end{align*}
	By a union bound, we have
	\begin{align*}
		&\mathbb{P}\left(\norm{G_{j_1+1}}_\infty> \frac{1}{\alpha}\right)\\
		\le&  \frac{2n^2\alpha}{\sqrt{2\pi}}\exp\left(-\frac{1}{2\alpha^{2}}\right)\\
		\le&\exp\left(-\frac{20}{\alpha}\right)\\
		\le&\alpha^{20}.
	\end{align*}
	By another union bound, with probability at least $1-\left(j_1+2\right)\alpha^{10}$, we have $\norm{G_{j_1+1}}_\infty\le \frac{1}{\alpha}$ and $$\sum_{i=1}^d \norm{K'^T_i w}_2^2\ge\varepsilon_{ j_2}^{10}$$ for any $w\in S^{n-1} \cap B_2$. 
	
	Under the above assumption, we will prove $\norm{K_{j_1+1}^T w}_2 \ge \varepsilon_{ j_2-1}^{\frac{7}{10}}$ for every $w$ in $\CompA{k_{j_2-1},l_{j_2-1},\varepsilon_{j_2-1}}$. Let $w$ be any such vector. By the definition of $B_0$, there exist $w'\in S^{n-1} \cap B_0$ with $\norm{w-w'}_2\le \varepsilon_{j_2}^7$. By orthogonal projection, there exist $w_1\in B_1$ and $w_2\in B_2$ with $w'=w_1+w_2$, $w_1^T w_2=0$ and $w_1^T K_{j_1}K_{j_1}^T w_2=0.$ Then we have
	\begin{align*}
		&\norm{K_{j_1+1}^T w}_2\\
		\ge&\norm{K_{j_1+1}^T w'}_2-\norm{K_{j_1+1}^T \left(w-w'\right)}_2\\
		\ge&\max\left\{\norm{K_{j_1}^T w'}_2, \sqrt{\sum_{i=1}^d\norm{K'^T_i w'}_2^2}\right\}-\sigma_1\left(K_{j_1+1}\right)\norm{w-w'}_2\\
		\ge&\max\left\{\norm{K_{j_1}^T w_1}_2, \sqrt{\sum_{i=1}^d\norm{K'^T_i w_2}_2^2}-\sqrt{\sum_{i=1}^d\norm{K'^T_i w_1}_2^2}\right\}-n\norm{G_{j_1+1}}_\infty \varepsilon_{j_2}^7\\
		\ge&\max\left\{\varepsilon_{j_2}^{\frac{7}{10}}\norm{w_1}_2,\varepsilon_{j_2}^{5 }\norm{w_2}_2-\sigma_1\left(K_{j_1+1}\right)\norm{w_1}_2\right\}-\alpha^{-2}\varepsilon_{j_2}^7\\
		\ge&\max\left\{\varepsilon_{j_2}^{\frac{7}{10}}\norm{w_1}_2, \varepsilon_{j_2}^{5}\norm{w_2}_2-\alpha^{-2}\norm{w_1}_2\right\}-\alpha^{-2}\varepsilon_{j_2}^7\\
		\ge&\frac{1}{2}\alpha^2\varepsilon_{j_2}^{\frac{57}{10}}-\alpha^{-2}\varepsilon_{j_2}^7\\
		\ge&\varepsilon_{j_2}^7\\
		=&\varepsilon_{j_2-1}^\frac{7}{10}.
	\end{align*}
		
	By mathematical induction, Lemma \ref{lemma-mid} holds.
	\end{proof}

	We are ready to prove Theorem \ref{main-4}.
	\begin{proof}[Proof of Theorem \ref{main-4}]
		 Take $c_{1.10}=\max\left\{c_{4.4},c_{4.6}, 10000\right\}$, where $c_{4.4}$ and $c_{4.6}$ are the absolute constants in Lemma \ref{lemma-d} and Lemma \ref{lemma-mid}. For each positive integer $i$ with $i\le\frac{n}{m}$, let $G_{\setminus i}$ be the $n\times \left(\frac{n}{m}-1\right)$ matrix obtained by deleting the $i$-th column from $G$. By Lemma \ref{lemma-mid}, with probability at least $1-n \alpha^{10}$, the matrix
		 $$K_{\setminus i}:=\left[
		 \begin{array}{c|c|c|c}
		 	A_1 G_{\setminus i} & A_2 G_{\setminus i} & \ldots & A_m G_{\setminus i}
		 \end{array}
		 \right]$$
		 satisfies 
		 $$\norm{K_{\setminus i}^T w}_2 \ge \varepsilon_{ 0}^{\frac{7}{10}}$$
		 for every vector $w\in \CompA{k_0,l_0,\varepsilon_0}$. Therefore, by Lemma \ref{lemma-d}, we have
		 \begin{align*}
		 	   &\mathbb{P}\left(\sigma_n(K)\le \alpha^{500m}\right)\\
		 	\le&\mathbb{P}\left(\sigma_n(K)\le \varepsilon_0^5\right)\\
		 	\le&\alpha^{7 k_0} +n^2\alpha^{10}\\
		 	=&\alpha^{7 m} +n^2\alpha^{10}\\
		 	\le&\alpha
		 \end{align*}
	 and Theorem \ref{main-4} holds.
	\end{proof}

	\section*{Acknowledgments}
	I thank Heng Liao and Bai Cheng for supporting me to work on this problem, and Merouane Debbah, Jiaoyang Huang and Richard Peng for helpful discussions. I also thank Ruixiang Zhang for pointing out the reference \cite[Lemma 4.7]{Guo} and Van Vu for suggestions on references.

\end{document}